\documentclass[british]{amsart}
\usepackage[T1]{fontenc}
\usepackage[utf8]{inputenc}
\usepackage[dvipsnames]{xcolor}
\usepackage[numbers]{natbib}
\usepackage{amstext,amsthm,amssymb}
\usepackage{xparse,mathrsfs,mathtools}
\usepackage{csquotes}
\usepackage{bookmark}
\usepackage{booktabs}
\usepackage[capitalise,nameinlink]{cleveref}
\usepackage{babel}
\usepackage[yyyymmdd,hhmmss]{datetime}
\usepackage{hyperref}
\usepackage{lmodern,microtype}

\hypersetup{
	unicode=true,%
	pdfdisplaydoctitle=true,%
	pdfpagemode=UseOutlines,%
	breaklinks=true,%
	pdfencoding=unicode,psdextra,
	pdfcreator={},
	pdfborder={0 0 0},
	hidelinks
}

\newtheorem{thm}{Theorem}[section]
\newtheorem{prop}[thm]{Proposition}
\newtheorem{lem}[thm]{Lemma}
\newtheorem{cor}[thm]{Corollary}
\newtheorem*{thm*}{Theorem}
\newtheorem*{lem*}{Lemma}
\theoremstyle{definition}
\newtheorem{rem}[thm]{Remark}
\newtheorem*{rem*}{Remark}

\DeclarePairedDelimiter\parentheses{\lparen}{\rparen}
\DeclarePairedDelimiter\floor{\lfloor}{\rfloor}
\DeclarePairedDelimiter\ceil{\lceil}{\rceil}
\DeclarePairedDelimiter\braces{\lbrace}{\rbrace}
\DeclarePairedDelimiter\abs{\lvert}{\rvert}
\DeclarePairedDelimiter\norm{\lVert}{\rVert}
\DeclarePairedDelimiter\ropeninterval{\lbrack}{\rparen}
\DeclarePairedDelimiter\brackets{\lbrack}{\rbrack}
\DeclarePairedDelimiter{\inner}{\langle}{\rangle}

\DeclareMathOperator{\area}{area}

\DeclareMathOperator{\diam}{diam}

\newcommand{\Norm}{\mathrm{N}}
\newcommand{\numberfieldK}{\mathbb{K}}

\newcommand{\integers}{\mathcal{O}}
\newcommand{\NN}{\mathbb{N}}
\newcommand{\ZZ}{\mathbb{Z}}
\newcommand{\QQ}{\mathbb{Q}}
\newcommand{\RR}{\mathbb{R}}
\newcommand{\CC}{\mathbb{C}}
\newcommand{\Primes}{\mathbb{P}}
\newcommand{\ReO}[1]{\Re_{#1}}
\newcommand{\ImO}[1]{\Im_{#1}}
\newcommand{\cupdot}{\mathbin{\mathaccent\cdot\cup}}
\newcommand{\dd}[1]{\mathop{\mathrm{d}#1}}
\newcommand{\ConditionalOne}[1]{{\boldsymbol{1}}_{\braces{#1}}}

\newcommand{\mrestrictedsum}[1]{\mathop{\sum\nolimits^{*}\!\sum}_{#1}}
\newcommand{\MNsum}{\mathop{\sum\!\sum}_{m,\,n}}
\newcommand{\JJsum}{\mathop{\sum\!\sum}_{\substack{ 1\leq \abs{j_1} < J \\ 1 \leq \abs{j_2}< J }}}
\newcommand{\bilinError}{E}
\DeclareMathOperator{\eOpname}{e}
\NewDocumentCommand\e{ s O{} m }{
	\IfBooleanTF{#1}{%
		\eOpname_{#2}\parentheses[\big]{#3}%
	}{\eOpname_{#2}\parentheses{#3}}%
}
\newcommand{\newUpperBound}{U}

\numberwithin{equation}{section}
\crefformat{equation}{#2(#1)#3}
\crefformat{enumi}{#2(#1)#3}

\usepackage[dvipsnames]{xcolor}
\usepackage{tikz}

\usetikzlibrary{calc}
\usetikzlibrary{arrows}
\usetikzlibrary{backgrounds}
\usetikzlibrary{positioning}
\usetikzlibrary{decorations.pathreplacing}

\definecolor{othercolor}{gray}{0.80}

\definecolor{othercolorTwo}{rgb}{1,1,1}

\title[On the distribution of $\alpha p$ mod.\ one in imaginary quadratic number fields]{On the distribution of 
$\alpha p$ modulo one in imaginary quadratic number fields with class number one}


\author{Stephan~Baier}
\address{Stephan~Baier\\%
	Ramakrishna Mission Vivekananda Educational Research Institute\\%
	Department of Mathematics\\%
	G.\ T.\ Road, PO~Belur Math, Howrah, West Bengal~711202\\%
	India}
\email{email\_baier@yahoo.de}
\urladdr{https://www.researchgate.net/profile/Stephan\_Baier2}

\author{Marc~Technau}
\address{Marc~Technau\\%
	Graz University of Technology\\%
	Institute of Analysis and Number Theory\\%
	Kopernikusgasse~24/II\\%
	8010~Graz\\%
	Austria}
\email{mtechnau@math.tugraz.at}
\urladdr{https://www.math.tugraz.at/\string~mtechnau/}

\begin{document}
\maketitle

	

\begin{abstract}
	We investigate the distribution of $\alpha p$ modulo one in imaginary quadratic number fields $\numberfieldK\subset\CC$ with class number one, 
	where $p$ is restricted to prime elements in the ring of integers $\integers = \ZZ[\omega]$ of $\numberfieldK$. In analogy to 
	classical work due to R.\ C.\ Vaughan, 
	we obtain that the inequality $\norm{\alpha p}_\omega < \Norm(p)^{-1/8+\epsilon}$ is satisfied for infinitely many $p$, 
	where $\norm{\varrho}_\omega$ measures the distance of $\varrho\in\CC$ to $\integers$ and $\Norm(p)$ denotes the norm of $p$.
	
	The proof is based on Harman's sieve method and employs number field analogues of classical ideas due to Vinogradov. 
	Moreover, we introduce a smoothing which allows us to make conveniently use of the Poisson summation formula.
\end{abstract}

\bigskip

\section{Introduction}
Dirichlet's classical approximation theorem asserts that, given some real irrational $\alpha$, there are infinitely many rational integers $a,q$ ($q\neq 0$) with
\[
	\abs*{ \alpha - a/q } < q^{-2},
\]
or---equivalently---on writing $\norm{\rho} = \min_{x\in\ZZ} \abs{ \rho - x }$ for the distance to a nearest integer,
\begin{equation}\label{eq:DirichletApprox}
	\norm{ q\alpha } < q^{-1}
	\quad\text{for infinitely many } q.
\end{equation}
Albeit individual values of $\alpha$ may allow for significantly sharper approximation by rational numbers, Hurwitz's approximation theorem implies that the exponent~$-1$ in~\cref{eq:DirichletApprox} is optimal in the sense that it cannot be decreased without the resulting new inequality failing to admit infinitely many solutions for some real irrational $\alpha$ (see, e.g., \cite[Theorems~193 and~194]{hardy2008anintroductionto}).

A natural variation on the question about the solubility of~\cref{eq:DirichletApprox} is to impose the additional restriction that $q$ be a rational prime and ask for which exponent $\theta$ one is able to establish that, for any real irrational $\alpha$,
\begin{equation}\label{eq:DirichletApprox:PrimeConstraint}
	\norm{ p\alpha } < p^{-\theta} \quad
	\text{for infinitely many rational primes }p.
\end{equation}
In this direction I.\ M.\ Vinogradov~\cite{vinogradov2004themethod} obtained~\cref{eq:DirichletApprox:PrimeConstraint} with $\theta = \frac{1}{5}-\epsilon$, a result which has since then been improved by a number of researchers (see \cref{tab:PrimeApprox:ExponentImprovements}) culminating in the work of \citet{matomaki2009the-distribution} who obtained $\theta=1/3-\epsilon$. This exponent is considered to be the limit of the current technology (see the comments in~\cite{heath-brown2002the-distribution}).
\begin{table}[!ht]
	\centering%
	\begingroup%
	\newcommand{\historyLine}[3]{ \citeyear{#1} & \citet{#1} & {#2} & {#3} \\ }
	\begin{tabular}{llll}
		\toprule
		Date & Author(s) & $\theta$ \\
		\cmidrule(r){1-1} \cmidrule(lr){2-2} \cmidrule(l){3-4}
		\historyLine{vaughan1978onthedistribution}{$1/4-\epsilon$}{$=0.25-\epsilon$}
		\historyLine{harman1983on-the-distribution}{$3/10$}{$=0.3$}
		\historyLine{jia1993on-the-distribution}{$4/13$}{$=0.3076\!\ldots$}
		\historyLine{harman1996on-the-distribu}{$7/22$}{$=0.31\overline{81}$}
		\historyLine{jia2000on-the-distribution}{$9/28$}{$=0.3214\!\ldots$}
		\historyLine{heath-brown2002the-distribution}{$16/49$}{$=0.3265\!\ldots$}
		\historyLine{matomaki2009the-distribution}{$1/3-\epsilon$}{$=0.3\overline{3}-\epsilon$}[0.5ex]
		\bottomrule
	\end{tabular}\vspace{1ex}
	\endgroup%
	\caption{Improvements on the admissible exponent $\theta$ in~\cref{eq:DirichletApprox:PrimeConstraint}.}%
	\label{tab:PrimeApprox:ExponentImprovements}%
\end{table}

In view of the above, the first named author~\cite{baier2016a-note-on} proposed to study the analogue of~\cref{eq:DirichletApprox:PrimeConstraint} for the Gaussian integers. The approach in~\cite{baier2016a-note-on} rests upon Harman's sieve method~\cite{harman1983on-the-distribution,harman1996on-the-distribu,harman2007primedetectingsieves} and the required \enquote{arithmetical input} is obtained using novel Gaussian integer analogues of classical ideas due to Vinogradov \cite[Lemma~8a]{vinogradov2004themethod}.

In this paper, we consider the more general problem of proving analogues of \cref{eq:DirichletApprox:PrimeConstraint} for imaginary quadratic number fields.
It turns out that, in our opinion, this setting also has the pleasant side effect of painting a clearer picture of the Diophantine arguments that underpin the aforementioned arithmetical information.
Preliminary results in this direction were obtained in the second author's doctoral dissertation~\cite{technau2018on-beatty}. A novel aspect of the present work is our additional use of smoothing directly incorporated into Harman's sieve method.

\section{Main results}
Before stating our results, we shall introduce some notation which is used throughout the rest of the article. We fix some imaginary quadratic number field $\numberfieldK$ with distinguished embedding into the complex numbers $\CC$ by means of which we shall regard $\numberfieldK$ as a subfield of $\CC$. By $\integers$ we denote the \emph{ring of integers} of $\numberfieldK$, i.e., the integral closure of $\ZZ$ in $\numberfieldK$. As $\numberfieldK$ is a quadratic extension of $\QQ$, it follows from well-known results from elementary algebraic number theory that $\integers$ is a free $\ZZ$-module of rank~$2$ and there is some $\omega\in\integers$ such that $\braces{1,\omega}$ is a $\ZZ$-basis of $\integers$. Since, by assumption, $\numberfieldK\nsubseteq\RR$, and $\numberfieldK$ being the field of fractions of $\integers$, it follows that $\Im\omega\neq 0$.
In particular, $\braces{1,\omega}$ turns out to be an $\RR$-basis of $\CC$ and, given some $\varrho\in\CC$, we write $\ReO{\omega}\varrho$ and $\ImO{\omega}\varrho$ for the unique real numbers satisfying
\[
	\varrho = \ReO{\omega}\varrho + \parentheses{\ImO{\omega}\varrho}\omega.
\]
With this notation, we put
\[
	\norm{\varrho}_\omega
	= \max\braces{ \norm{\ReO{\omega}\varrho}, \norm{\ImO{\omega}\varrho} }.
\]
The natural notion of \enquote{size} of an element $m\in\integers$ is furnished by its \emph{norm} $\Norm(m)$, that is, the number of elements in the factor ring $\integers/m\integers$. It can be shown that $\Norm(m) = \abs{m}^2$, where $\abs{m}$ is the usual absolute value of $m$ considered as a complex number.

The question we ask may now be enunciated as follows:
\begin{quote}\itshape
	Given some imaginary quadratic number field $\numberfieldK\subset\CC$ with ring of integers $\integers$, a choice $\braces{1,\omega}$ of $\ZZ$-basis of $\integers$, and given some $\alpha\in\CC\setminus\numberfieldK$, for which $\theta>0$, does one have
	\[
		\norm{p\alpha}_\omega < \Norm(p)^{-\theta}
	\]
	for infinitely many irreducible (or prime) elements $p\in\integers$?
\end{quote}


As unique factorisation underpins the sieve method we employ to tackle the above question, we are forced to restrict our considerations to only those $\numberfieldK$ with class number~$1$ (which, in this setting, is equivalent to $\integers$ being a unique factorisation domain).
The full determination of all such $\numberfieldK$ is provided by the celebrated \emph{Baker--Heegner--Stark theorem}~\cite{heegner1952diophantische-analysis,baker1966linear-forms,stark1967a-complete-determination,stark1969on-the-gap}:
\begin{thm*}[\citeauthor{baker1966linear-forms}--\citeauthor{heegner1952diophantische-analysis}--\citeauthor{stark1969on-the-gap}]
	The imaginary quadratic number fields $\numberfieldK$ with class number~$1$ are (up to isomorphism) precisely those $\QQ(\sqrt{d})$ with $-d$ from the finite list $1$, $2$, $3$, $7$, $11$, $19$, $43$, $67$, $163$.
\end{thm*}

Our main result states that in the above question any $\theta<1/8$ is admissible, 
provided that $\numberfieldK$ has class number~$1$. This is the precise analogue of Vaughan's exponent $1/4-\epsilon$ for the classical case,
obtained in \cite{vaughan1978onthedistribution}.
Note that in the class number~$1$ setting, the notions of \emph{prime} and \emph{irreducible} coincide.
\begin{thm}
	\label{thm:ApproxWithPrimes}
	Let $\numberfieldK\subset\CC$ be an imaginary quadratic number field with class number~$1$ and let $\integers$ be its ring of integers with $\ZZ$-basis $\braces{1,\omega}$.
	Suppose that $\alpha$ is a complex number such that $\alpha\notin\numberfieldK$.
	Then, for any $\epsilon>0$, there exists an infinite sequence of distinct prime elements $p\in\integers$ such that
	\begin{equation}\label{eq:-1/16:result}
		\norm{p\alpha}_\omega \leq \Norm(p)^{-1/8+\epsilon}.
	\end{equation}
\end{thm}

Our approach to proving \cref{thm:ApproxWithPrimes} involves counting prime elements with a certain smooth weight attached to them. On the other hand, one can also use sharp cut-offs and obtain a less fuzzy quantitative result at the cost of having to restrict to smaller values of $\theta$ in the above question. We prove the following:
\begin{thm}
	\label{thm:DiophApprox:Integers}
	Let $\numberfieldK\subset\CC$ be an imaginary quadratic number field with class number~$1$ and let $\integers$ be its ring of integers with $\ZZ$-basis $\braces{1,\omega}$.
	Suppose that $\epsilon>0$ is sufficiently small.
	Let $\alpha$ be a complex number not contained in $\numberfieldK$.
	Furthermore, suppose that one has coprime $a,q\in\integers$ such that
	\begin{equation}\label{eq:ThetaApprox}
		q\neq 0,\quad
		\frac{a}{q}\notin\integers,
		\quad\text{and}\quad
		\gamma = \alpha - \frac{a}{q}
		\quad\text{satisfies}\quad
		\abs{\gamma} \leq \frac{C}{\Norm(q)}
	\end{equation}
	for some constant $C>0$ and put $x = \Norm(q)^{28/5}$. Then, for any $\delta$ such that
	\begin{equation}\label{eq:DeltaRange}
		x^{-1/28+\epsilon} \leq \delta < \tfrac{1}{2},
	\end{equation}
	we have
	\[
		\abs[\bigg]{ \sum_{\substack{ x/2\leq \Norm(p) < x \\ \norm{p\alpha}_\omega < \delta }} 1
		- 4\delta^2 \sum_{x/2\leq \Norm(p) < x} 1 }
		\ll_\epsilon C^2 \Norm(\omega)^7 \delta^2 x^{1-\epsilon},
	\]
	where the summation variable $p$ (as throughout) only assumes prime elements of $\integers$ and the implied constant depends on $\epsilon$ alone.
\end{thm}
\begin{rem*}
	(1) Instead of considering the homogeneous condition $\norm{p\alpha}_\omega < \delta$ in the above theorem, one can also consider a shifted version, namely $\norm{p\alpha+\beta}_\omega < \delta$, where $\beta$ is an arbitrary complex number.
	In fact, the authors~\cite{vaughan1978onthedistribution,harman1983on-the-distribution,jia1993on-the-distribution,harman1996on-the-distribu,jia2000on-the-distribution} listed in \cref{tab:PrimeApprox:ExponentImprovements} also consider the shifted analogue of~\cref{eq:DirichletApprox:PrimeConstraint}, but the innovation introduced by \citet{heath-brown2002the-distribution} has, as they remark, the defect of entailing the restriction to $\beta=0$.
	Regardless of this, the methods pursued by us in the present paper are perfectly capable of handling shifts $\beta$ and we merely chose not to implement this for cosmetic reasons.
	
	(2) In his recent preprint \cite{Harmanpreprint} on the case $\numberfieldK=\mathbb{Q}(i)$, 
	Harman achieved the result in \cref{thm:DiophApprox:Integers} with the exponent $7/44$,
	which corresponds to the exponent 7/22 in his classical result \cite{harman1996on-the-distribu} mentioned above. 
	To this end, he didn't use a
	smoothing but introduced a number of novelties to overcome obstacles that are present in the non-smoothed approach. 
	In particular, he was able to handle linear exponential sums over certain regions in $\mathbb{C}$ in an efficient way. It is likely that these novelties can be carried over to all imaginary-quadratic
	fields of class number 1, but we here confine ourselves to the simplest possible treatment, thus obtaining the exponent 1/28 for all fields of this kind. The way we overcome the
	said obstacles is to introduce a smoothing which allows us to use the Poisson summation formula conveniently. 
	This leads us to \cref{thm:ApproxWithPrimes}, where we achieve the exponent 1/8 
	corresponding to Vaughan's result for the classical case 
	\cite{vaughan1978onthedistribution}. To achieve the larger exponent $7/44$ in full generality for the said fields, it would be 
	required to carry over Harman's lower bound sieve, established in \cite{harman1996on-the-distribu}, to them. This is a task we aim to undertake in a separate paper since the proof of the aforementioned 
	sieve result is technically complicated and therefore requires a large amount of extra work. In particular, it involves a number of numerical calculations which are not required for the proof of the basic version
	which we are using here.
\end{rem*}

Still assuming $\numberfieldK$ to have class number~$1$, and appealing to Landau's prime ideal theorem one easily deduces that
\begin{equation}\label{eq:NumberFieldPNT}
	\sum_{x/2\leq \Norm(p) < x} 1
	= \delta_\integers \frac{x}{\log x} (1+o_\integers(1))
	\quad\text{as } x\longrightarrow\infty,
\end{equation}
where $\delta_\integers>0$ is some constant only depending on $\integers$.

Therefore, \cref{thm:DiophApprox:Integers} implies \cref{thm:ApproxWithPrimes} with the exponent $-1/8$ in \cref{eq:-1/16:result} replaced with $-1/28$ provided one is able to verify the existence of infinitely many $a$ and $q$ as required by the theorem; however, the latter problem is already solved by Hilde~Gintner~\cite{gintner1936ueberkettenbr}. In this regard, let $\Lambda$ be the \emph{fundamental parallelogram} spanned by $1$ and $\omega$,
\begin{equation}\label{eq:FundamentalParallelogram}
	\Lambda = \braces{ \lambda_1 + \lambda_2\omega : \lambda_1,\lambda_2\in\ropeninterval{0,1} }.
\end{equation}
\begin{lem*}[\citeauthor{gintner1936ueberkettenbr}]
	Let $\alpha$ be a complex number not contained in $\numberfieldK$. Then there are infinitely many $a,q\in\integers$ satisfying~\cref{eq:ThetaApprox} with
	\(
		C = \pi^{-1}\sqrt{6} \area\Lambda
	\)
	and $\Lambda$ given by~\cref{eq:FundamentalParallelogram}.
\end{lem*}
Albeit the above lemma does not assert that $a,q$ be coprime, if $\numberfieldK$ has class number~$1$, then one can appeal to unique factorisation and cancel any potential non-trivial common factors from $a$ and $q$. So, indeed, one has the aforementioned relation between \cref{thm:DiophApprox:Integers} and \cref{thm:ApproxWithPrimes}.

\section{Outline of the method}

\subsection{The sieve method}
For the detection of the prime elements in \cref{thm:ApproxWithPrimes} and \cref{thm:DiophApprox:Integers} we use a sieve result due to Harman---with additional smoothing and adapted to our number field setting---which has the pleasant feature of keeping our exposition reasonably tidy.
The underlying sieve method itself and its various refinements are also capable of yielding lower bounds instead of asymptotic formulae, in exchange for the prospect of increasing the admissible range for $\delta$ in~\cref{eq:DeltaRange} (or $\theta$ in the main question), but we do not implement this here.
The interested reader is referred to Harman's exposition of his method~\cite{harman2007primedetectingsieves}.
The following special case suffices for our purposes; we write $d_k(\mathfrak{a})$ for the number of ways in which an ideal $\mathfrak{a}\subseteq\integers$ can be written as a product of $k$ ideals of $\integers$, and we write $d(\mathfrak{a}) = d_2(\mathfrak{a})$.

\begin{thm}[Weighted version of Harman's sieve for $\integers$]
	\label{thm:HarmansSieve}
	Suppose that $\integers$ has class number~$1$ and let $x\geq3$ be real.
	Let $w,\tilde{w}\colon\integers \to \brackets{0,1}$ be two functions such that, for both $\omega=w$ and $\omega=\tilde{w}$,
	\begin{equation}\label{eq:Convergence}
		\lim_{R\to\infty} \sum_{\substack{ r\in\integers \\ \mathclap{\Norm(r)<R} }} d_4(r\integers) \omega(r) \leq X
	\end{equation}
	for some $X\geq 1$ and assume that $\omega(r)=\omega(\tilde{r})$ for any pair of associate elements $r,\tilde{r}\in\integers$.
	Suppose further that one has numbers $Y > 1$, $0<\mu<1$, $0 < \kappa \leq \tfrac{1}{2}$, and $M \in (x^{\mu},x)$ with the following property:\\
	For any sequences $(a_{m})_{m\in\integers}$, $(b_{n})_{n\in\integers}$ of complex numbers with $\abs{a_{m}}\leq 1$ and $\abs{b_{n}}\leq d\parentheses{n\integers}$, one has
	\begin{gather}
		\label{eq:Harman:Type:I}
		\abs[\bigg]{
			\mathop{\sum\!\sum}_{\substack{ m,n\in\integers\setminus\braces{0} \\ \mathclap{0< \Norm(m)<M }}} a_m (w(mn) - \tilde{w}(mn))
		} \leq Y, \\
		\label{eq:Harman:Type:II}
		\abs[\bigg]{
			\mathop{\sum\!\sum}_{\substack{ m,n\in\integers\setminus\braces{0} \\ \mathclap{ x^\mu < \Norm(m) < x^{\mu+\kappa} }}} a_m b_n (w(mn) - \tilde{w}(mn))
		} \leq Y.
	\end{gather}
	Then
	\[
		\abs{ S(w,x^{\kappa}) - S(\tilde{w},x^{\kappa}) } \ll Y (\log(xX))^3,
	\]
	where
	\begin{equation}\label{eq:SiftingFunction}
		S(\omega,x^{\kappa}) = \sum_{\substack{ r\in\integers\setminus\braces{0} \\ \text{\rm prime } p\mid r \Rightarrow \Norm(p)\geq x^\kappa }} \omega(r),
	\end{equation}
	and the implied constant is absolute.
	(Here all infinite series appearing in \cref{eq:Harman:Type:I,eq:Harman:Type:II,eq:SiftingFunction} are guaranteed to be absolutely convergent by~\cref{eq:Convergence}.)
\end{thm}
\begin{rem}\label{rem:OverallProcedure}
	We comment briefly on how we apply the above theorem: informally speaking, the goal is to choose the weight functions $w$ and $\tilde{w}$ in such a way that essentially only elements $r$ with $\Norm(r)<x$ contribute in the definition of $S(w,x^\kappa)$ and $S(\tilde{w},x^\kappa)$. Then, choosing $\kappa=\frac{1}{2}$, these $r$ are guaranteed to be prime, and, for $\omega\in\braces{w,\tilde{w}}$,
	\[
		S(\omega,\sqrt{x}) \approx \sum_{\substack{ \text{prime } p \in \integers \\ \sqrt{x} \leq \Norm(p)<x }} \omega(p).
	\]
	The choice of $w$ is made as to guarantee that $S(w,\sqrt{x})$ is essentially a known quantity by an appeal to~\cref{eq:NumberFieldPNT}, and $\tilde{w}$ is tailored to enforce a restriction such as $\norm{p\alpha}_\omega \leq \delta$ as in \cref{thm:DiophApprox:Integers}.
	Finally, assuming that one proves suitably strong versions of~\cref{eq:Harman:Type:I} and~\cref{eq:Harman:Type:II}, \cref{thm:HarmansSieve} asserts that $S(\tilde{w},\sqrt{x})$ must be of similar magnitude as $S(w,\sqrt{x})$, and, therefore, one ascertains information about the abundance of prime elements with the desired properties as encoded in the weight function $\tilde{w}$.
\end{rem}

The proof of \cref{thm:HarmansSieve} is essentially identical to the usual proof of Harman's sieve in the setting of $\ZZ$~\cite[\S\S~3.2--3.3]{harman2007primedetectingsieves}; it uses the sieve of Eratosthenes--Legendre to relate $S(w,x^\kappa) - S(\tilde{w},x^\kappa)$ to certain sums with more variables and applies Buchstab's identity multiple times to produce variables in the correct ranges for \cref{eq:Harman:Type:I,eq:Harman:Type:II} to become applicable. In the process of doing so, one has to remove certain cross-conditions between summation variables---a feat which is accomplished using a variant of Perron's formula (see \cref{lem:PerronVariant} below).
In our number field setting, this results in a slight complication which is not present when working over the rational integers: at some point one is faced with having to remove a condition of the type $\Norm(p) \leq \Norm(\tilde{p})$ from a double sum with summation variables $p$ and $\tilde{p}$ assuming only non-associate prime elements as values (see the arguments around~\cref{eq:S_II_2:Decomposition} below).
However, this problem can be overcome.
For the convenience of the reader we provide a detailed proof of \cref{thm:HarmansSieve} in \cref{sec:HarmansSieveProof} below.

\subsection{Outline of the rest of the paper}
Apart from proving \cref{thm:HarmansSieve} in \cref{sec:HarmansSieveProof}, we proceed as follows: by the outline given in \cref{rem:OverallProcedure}, the bulk of the remaining work lies in the verification of~\cref{eq:Harman:Type:I} and~\cref{eq:Harman:Type:II} for the two choices of $w$ and $\tilde{w}$ that we use below (non-smoothed and smoothed). Both arguments ultimately hinge on distribution results related to the sequence $\norm{n\alpha}_\omega$ ($n\in\integers$). We devote our attentions to establishing such results first. This is done in \cref{sec:ExponentialSumEstimates}, with its principal results being stated in \cref{thm:ExpSumBound} and \cref{thm:Diophantine}. \\
\Cref{sec:HarmanInput} is devoted to proving \cref{thm:DiophApprox:Integers}, with this goal being achieved in \cref{subsec:AssemblingTheParts}. The results concerning~\cref{eq:Harman:Type:I} and~\cref{eq:Harman:Type:II} are recorded in \cref{prop:TypeI:ExplicitBound} and \cref{prop:TypeII:ExplicitBound} respectively. \\
In \cref{sec:Smoothed} we undertake proving \cref{thm:ApproxWithPrimes}. The analogues of the aforementioned propositions are \cref{prop:Type-I-1} and \cref{prop:Type-II-1} and their proof is largely parallel to the proof of their non-smoothed counterparts.
The main innovation here is contained in \cref{lem:PoissonResult-1}, which takes advantage of the smooth weights by means of Poisson's summation formula.

\section{Exponential sum estimates}
\label{sec:ExponentialSumEstimates}

In \cref{sec:HarmanInput} we need estimates for sums of the shape
\[
	\sum_n \sum_m \e{ \ImO{\omega}(mn\alpha) },
\]
where the summation over $m\in\integers$ is restricted to some annulus $x(n)\leq \Norm(m) \leq y(n)$ with bounds $x(n)$ and $y(n)$ depending on $n\in\integers$. In \cref{subsec:LinearExpSumBound} we estimate the inner summation and in \cref{subsec:DistribFracBounds} we deal with the additional summation over $n$. The corresponding proof is then carried out in the subsequent sections. Moreover, using the tools developed in \cref{sec:DiophantineLemmas}, we also establish a closely related result which is useful in \cref{sec:Smoothed} (see \cref{thm:Diophantine}).

In what follows, we sometimes have expressions like $1/\norm{ \ReO{\omega}\alpha }$; if a division by zero occurs there, then the result is understood to mean $+\infty$.

\subsection{Some facts about quadratic extensions}
Before being able to tackle the problem outlined above, we take the opportunity to record here some basic facts about quadratic extensions which we use throughout.
\begin{lem}
	\label{lem:AlgebraicFacts}
	Let $\numberfieldK\subset\CC$ be an imaginary quadratic number field with ring of integers $\integers$ and generator $\omega$ of $\integers$, that is, $\integers = \ZZ[\omega]$. Then the following statements hold:
	\begin{enumerate}
		\item\label{enu:UnitsInRingOfIntegers} The number of units in $\integers$ is bounded by six.
		\item\label{enu:ImOmegaLowerBound} $\abs{\Im\omega} \geq \sqrt{3}/2$.
		\item\label{enu:2ReOmegaIsIntegral} $2\Re\omega$ is an integer.
		\item\label{enu:LatticeCount:UpperBound} For $y\geq 1$, $\#\braces{ n\in\integers : \Norm(n) \leq y} \ll y$, where the implied constant is absolute.
	\end{enumerate}
\end{lem}
\begin{proof}
	The first assertion may be found in \cite{hardy2008anintroductionto}.
	
	For the second assertion just observe that there is some negative square-free rational integer $d$ such that $\numberfieldK = \QQ(\sqrt{d})$. Letting $D = 4d$ if $d\not\equiv 1\bmod 4$, and $D = d$ otherwise, we have $\integers = \ZZ[\tilde{\omega}]$, where $\tilde{\omega} = \frac{1}{2}(D+\sqrt{D})$. An elementary calculation shows that $\omega$ is of the form $k\pm{\tilde\omega}$ for some rational integer $k$. This already proves~\cref{enu:2ReOmegaIsIntegral}. Moreover, from this and a short computation one immediately obtains~\cref{enu:ImOmegaLowerBound} (with equality being attained for $d=-3$).
	
	Concerning the last assertion, we note that the quantity bounded therein, $\#\braces{ (n_1,n_2)\in\ZZ^2 : \abs{n_1+n_2\omega}^2 \leq y }$, counts points inside some ellipse. Elementary arguments suffice to show that this is asymptotically equal to $\pi y/\abs{\Im\omega}$ and the uniform lower bound for $1/\abs{\Im\omega}$ furnished by the second assertion finishes the proof.
\end{proof}

\subsection{Basic estimates for linear exponential sums}
\label{subsec:LinearExpSumBound}

\begin{lem}
	\label{lem:LinearExpSumBound}
	Let $\alpha$ be a complex number and suppose that one has numbers $x,y$ such that $0\leq x\leq y$. Then
	\begin{equation}\label{eq:LinearExpSumBound}
		\abs[\bigg]{
			\sum_{x\leq \Norm(m) \leq y} \e{ \ImO{\omega}(m\alpha) }
		} \ll \Norm(\omega) \sqrt{y} \min\braces*{
			\sqrt{y}, \frac{1}{\norm{ \ReO{\omega}\alpha }}, \frac{1}{\norm{ \ImO{\omega}\alpha }}
		}.
	\end{equation}
\end{lem}
\begin{proof}
	\begingroup
	\newcommand{\linearSum}[2]{\operatorname{Lin}(#1,#2)}
	We may assume $y\geq 1$, for~\cref{eq:LinearExpSumBound} is trivial otherwise.
	We denote the sum on the left hand side of~\cref{eq:LinearExpSumBound} by $\linearSum{x}{y}$. On writing
	\begin{equation}\label{eq:OmegaSquared:XiDecomp}
		\omega^2 = \xi_1 + \xi_2\omega
		\quad (
			\xi_1 = \ReO{\omega}\omega^2, \quad
			\xi_2 = \ImO{\omega}\omega^2
		),
	\end{equation}
	$m=m_1 +m_2\omega$ and $\ell_1 =m_1 +m_2\xi_2$, we obtain the following two expressions for $\ImO{\omega}\parentheses{m\alpha}$:
	\begin{align*}
		\ImO{\omega}\parentheses{m\alpha} &
		= m_2 \parentheses{\ReO{\omega}\alpha + \xi_2\ImO{\omega}\alpha} + m_1 \ImO{\omega}\alpha 
		= \ell_1 \ImO{\omega}\alpha + m_2\ReO{\omega}\alpha.
	\end{align*}
	Therefore,
	\begin{equation}\label{eq:EllipticalDiskSlicing1}
		\linearSum{0}{y} = 
		\mathop{\sum_{m_1}\sum_{m_2}}_{\mathclap{0\leq \Norm(m_1+m_2\omega)\leq y}} \e{ m_2\parentheses{ \ReO{\omega}\alpha + \xi_2\ImO{\omega}\alpha } } \e{ m_1\ImO{\omega}\alpha }
	\end{equation}
	and
	\begin{equation}\label{eq:EllipticalDiskSlicing2}
		\linearSum{0}{y} = 
		\mathop{\sum_{\ell_1}\sum_{m_2}}_{\mathclap{0\leq\Norm(\ell_1-m_2\xi_2+m_2\omega)\leq y}} \e{\ell_1\ImO{\omega}\alpha} \e{m_2\ReO{\omega}\alpha}.
	\end{equation}
	So, on recalling the well-known bound
	\[
		\abs[\bigg]{ \sum_{a\leq j\leq b} \e{j\rho} }
		\leq \frac{1}{2\norm{\rho}}
		\quad (a,b,\rho\in\RR),
	\]
	and using the triangle inequality on the outer summations in~\cref{eq:EllipticalDiskSlicing1} and \cref{eq:EllipticalDiskSlicing2},
	we obtain
	\begin{gather}
		\label{eq:LinearExpSumBound:Im}
		\abs{\linearSum{0}{y}}
		\leq \frac{ \#\braces{m_2 : \exists m_1 \text{ s.t. }0\leq \Norm(m_1 +m_2\omega) \leq y} }{ 2\norm{ \ImO{\omega}\alpha }  }, \\
		\label{eq:LinearExpSumBound:Re}
		\abs{\linearSum{0}{y}}
		\leq \frac{ \#\braces{\ell_1  : \exists m_2\text{ s.t. }0\leq \Norm(\ell_1 -m_2\xi_2+m_2\omega) \leq y} }{ 2\norm{ \ReO{\omega}\alpha } }.
	\end{gather}
	The numerators here are bounded easily: indeed, since
	\[
		\Norm(m_1+m_2\omega)
		\geq m_2^2 (\Im\omega)^2,
	\]
	using $y\geq1$ and \cref{lem:AlgebraicFacts}~\cref{enu:ImOmegaLowerBound}, one easily bounds the numerator in~\cref{eq:LinearExpSumBound:Im} by $4\sqrt{y}$.
	On the other hand for $\abs{\ell_1 }>c\sqrt{y}$ with $c = \parentheses{2+\frac{2}{\sqrt{3}}} \Norm(\omega)$ there is no $m_2$ such that
	\begin{equation}\label{eq:NormMaxCoordLowerBound}
		\begin{aligned}
			y & \geq \Norm(\ell_1 -m_2\xi_2+m_2\omega) \\ &
			\geq \max\braces{ \abs{\ell_1 -m_2 \xi_2 + m_2\Re\omega}^2,\, m_2^2 (\Im\omega)^2 },
		\end{aligned}
	\end{equation}
	for otherwise it would follow that
	\[
		\sqrt{y} > c\sqrt{y} - \abs{m_2} \cdot \abs{\xi_2-\Re\omega}
	\]
	so that $\abs{m_2} \cdot \abs{\xi_2-\Re\omega} > \parentheses{c-1} \sqrt{y}$, but then
	\begin{align*}
		\abs{ m_2 \Im\omega } & > \frac{c-1}{\abs{\xi_2}+\abs{\Re\omega}} \sqrt{y}
		\geq \frac{\parentheses{1+2/\sqrt{3}}\abs{\omega}^2 }{\abs{\omega}^2 +\abs{\omega}} \sqrt{y} 
		= \frac{1+2/\sqrt{3}}{1+\abs{\omega}^{-1}} \sqrt{y}
		\geq \sqrt{y},
	\end{align*}
	in contradiction to~\cref{eq:NormMaxCoordLowerBound}. Hence, the numerator in~\cref{eq:LinearExpSumBound:Re} is bounded by $2c\sqrt{y}+1\leq 8\abs{\omega}^2 \sqrt{y}$. Thus,
	\[
		\abs{\linearSum{0}{y}} \leq 8 \abs{\omega}^2 \sqrt{y} \min\braces*{
			\frac{1}{2\norm{\ImO{\omega}\alpha}},
			\frac{1}{2\norm{\ReO{\omega}\alpha}}
		}
	\]
	and, together with the trivial bound $\abs{\linearSum{x}{y}} \ll y$ from \cref{lem:AlgebraicFacts}~\cref{enu:LatticeCount:UpperBound}, this is clearly satisfactory to establish~\cref{eq:LinearExpSumBound} for $x=0$, and the case $x>0$ follows from this bound and
	\[
		\linearSum{x}{y} = \linearSum{0}{y} - \lim_{\epsilon\searrow0} \linearSum{0}{x-\epsilon}.
		\qedhere
	\]
	\endgroup
\end{proof}

\subsection{Distribution of fractional parts}
\label{subsec:DistribFracBounds}

In view of \cref{lem:LinearExpSumBound} and recalling the goal stated at the beginning of \cref{sec:ExponentialSumEstimates}, we are faced with the problem of estimating sums of the shape
\begin{equation}\label{eq:E:average}
	\sum_{n\in\mathscr{X}} E(n,M),
\end{equation}
where $M\geq 2$, $\mathscr{X}$ is some subset of $\braces{ n\in\integers : 1\leq \Norm(n) < x}$ with $x\geq 1$, and
\begin{equation}\label{eq:LinearSumBound}
	E(n,M) = \min\braces*{
		M,
		\frac{1}{\norm{ \ReO{\omega}(n\alpha) }},
		\frac{1}{\norm{ \ImO{\omega}(n\alpha) }}
	}.
\end{equation}
The usual attack against such a problem is to replace $\alpha$ by some Diophantine approximation $a/q$. Subsequently, after bounding the error introduced from the approximation, one is able to control averages of $E(n,M)$ with $n$ constrained to boxes (say) not too large in terms of $\abs{q}$.
By splitting the full range of $n$ in~\cref{eq:E:average} into such boxes, one derives a bound for~\cref{eq:E:average} of the shape seen in \cref{thm:ExpSumBound} below.
\begin{thm}
	\label{thm:ExpSumBound}
	Let $\mathscr{X}$ be a subset of all $n\in\integers$ with $1\leq \Norm(n)<x$ and suppose that one has coprime $a,q\in\integers$ satisfying~\cref{eq:ThetaApprox}.
	Put
	\[
		S = \sum_{n\in\mathscr{X}} \min\braces*{
			M,
			\frac{1}{\norm{ \ReO{\omega}(n\alpha) }},
			\frac{1}{\norm{ \ImO{\omega}(n\alpha) }}
		}.
	\]
	Then, assuming $M\geq 2$,
	\begin{equation}\label{eq:ExpSumAverageBound:1}
		S \ll
		\parentheses{1 + C^2 \Norm(\omega^2) x/\Norm(q)}
		\parentheses{M + \Norm(q\omega) \log M}.
	\end{equation}
	Furthermore, if $x \leq \Norm(q) / \parentheses{12C \Norm(\omega)}^2$, then
	\begin{equation}\label{eq:ExpSumAverageBound:2}
		S \ll \Norm(q\omega) \log \Norm(q\omega).
	\end{equation}
\end{thm}
The proof of this result follows the outline given above and is undertaken in the next two subsections.
Similarly, we also obtain the following result:
\begin{thm}
	\label{thm:Diophantine}
	Suppose that one has coprime $a,q\in\integers$ satisfying~\cref{eq:ThetaApprox}. For $0<\Delta\leq\tfrac{1}{2}$ let
	\[
		H_{\alpha}(x,\Delta)
		= \#\braces{ n\in\integers : 0<\Norm(n)\leq x, \, \norm{n\alpha}_\omega\leq\Delta }.
	\]
	Then
	\[
		H_{\alpha}(x,\Delta) \ll_{C,\omega} \parentheses{ 1+x\abs{q}^{-2} } \parentheses{ 1+\Delta^2\abs{q}^2 }.
	\]
	Moreover, $H_{\alpha}(x,\Delta)$ vanishes if $\Delta < 1/(4\abs{q\omega})$ and $x\leq\abs{q}^2/(12C\abs{\omega}^2)^2$.
\end{thm}

\subsection{Diophantine lemmas}
\label{sec:DiophantineLemmas}

As a first step, we show that
\[
	\norm{\alpha}_\omega = \max\braces{ \norm{ \ReO{\omega}\alpha}, \norm{ \ImO{\omega}\alpha} }
\]
cannot be too small if $\alpha\in\numberfieldK$ is not an algebraic integer.
The results in this section are probably already known in one form or another.
However, we were unable to find a suitable reference and, therefore, provide full proofs for the reader's convenience.
\begin{lem}
	\label{lem:FractionalPartBounds}
	For non-zero $a,q\in\integers$ such that $\alpha = a/q\notin\integers$, it holds that
	$\norm{\alpha}_\omega \geq 1/(2\abs{q\omega})$.
\end{lem}
\begin{proof}
	Pick $m = m_1+m_2\omega\in\integers$ such that
	\[
		\norm{\ReO{\omega}\alpha} = \abs{\ReO{\omega}\alpha-m_1}
		\quad\text{and}\quad
		\norm{\ImO{\omega}\alpha} = \abs{\ImO{\omega}\alpha-m_2}.
	\]
	Now certainly it holds that
	\begin{align*}
		\abs{\alpha-m} &
		= \abs{\parentheses{\ReO{\omega}\alpha-m_1 } + \parentheses{\ImO{\omega}\alpha-m_2}\omega}  \\ &
		\leq 2\abs{\omega} \max\braces{
			\abs{\ReO{\omega}\alpha-m_1},
			\abs{\ImO{\omega}\alpha-m_2}
		} \\ &
		= 2 \abs{\omega} \max\braces{
			\norm{\ReO{\omega}\alpha},
			\norm{\ImO{\omega}\alpha}
		}.
	\end{align*}
	Therefore, to prove the lemma, it suffices to give a suitable lower bound for $\abs{\alpha-m}$, which, upon noting that $\alpha\notin\integers$, is quite easy:
	\[
		\abs{\alpha-m}
		= \abs{q}^{-1}\abs{a-qm}
		\geq \abs{q}^{-1} \min_{0\neq r\in\integers} \abs{r}
		= \abs{q}^{-1}.
		\qedhere
	\]
\end{proof}
Next, we intend to derive a result similar to \cref{lem:FractionalPartBounds}, when $\alpha$ is slightly perturbed:
\begin{lem}
	\label{lem:FractionalPartBounds:Pertubed}
	Let $\alpha$ be a complex number and $a,q$ be such that~\cref{eq:ThetaApprox} holds. Furthermore, suppose that $n\in\integers$
	satisfies $\abs{n} \leq \abs{q} / (12C\abs{\omega}^2)$
	and $na$ be indivisible by $q$. Then
	$\norm{n\alpha}_\omega \geq 1/(4\abs{q\omega})$.
\end{lem}
\begin{proof}
	First, we separate the perturbation $\gamma$ from the rest: we have
	\begin{align*}
		\norm{\ImO{\omega}(n\alpha)} &
		= \min_{k\in\ZZ} \abs{ \ImO{\omega}(na/q)-k+\ImO{\omega}\parentheses{n\gamma} } \\ &
		\geq \min_{k\in\ZZ} \abs{ \ImO{\omega}(na/q)-k } - \abs{\ImO{\omega}(n\gamma)} \\ &
		= \norm{\ImO{\omega}(na/q)} - \abs{\ImO{\omega}(n\gamma)},
	\end{align*}
	and the same holds when one replaces $\ImO{\omega}$ by $\ReO{\omega}$.
	The last term therein is bounded easily: using $\abs{\xi} \geq \abs{\ImO{\omega}\xi} \cdot \abs{\Im\omega}$, \cref{lem:AlgebraicFacts}~\cref{enu:ImOmegaLowerBound} and writing $N = C\abs{n}/\abs{q}^2$ for the moment, we have
	\[
		\abs{ \ImO{\omega}\parentheses{n\gamma} }
		\leq \max_{\substack{ \xi\in\CC \\ \abs{\xi}\leq N }} \abs{\ImO{\omega}\xi}
		\leq \frac{N}{\abs{\Im\omega}}
		\leq \frac{2N}{\sqrt{3}}.
	\]
	A similar calculation also bounds the corresponding $\ReO{\omega}$-term:
	\begin{align*}
		\abs{\ReO{\omega}\parentheses{n\gamma}} &
		\leq \max\braces[\big]{ \abs{\rho_1} :
			\rho_1,\rho_2\in\RR,\, (\rho_1 +\rho_2\Re\omega)^2 +(\rho_2\Im\omega)^2 \leq N^2
		} \\ &
		\leq \max\braces[\big]{
			\abs{\theta}+\abs{\rho_2\Re\omega} :
			\theta,\rho_2\in\RR, \,
			\theta^2 +(\rho_2\Im\omega)^2 \leq N^2
		} \\ &
		\leq \parentheses{1+\abs{\Re\omega/\Im\omega}} N \\ &
		\leq \parentheses{ 1+2/\sqrt{3} } \abs{\omega} N.
	\end{align*}
	Thus, using \cref{lem:FractionalPartBounds},
	\begin{align*}
		\norm{n\alpha}_\omega &
		= \max\braces{
			\norm{\ReO{\omega}\parentheses{n\alpha}},
			\norm{\ImO{\omega}\parentheses{n\alpha}}
		} \\ &
		\geq\max\braces{
			\norm{ \ReO{\omega}\parentheses{na/q}},
			\norm{ \ImO{\omega}\parentheses{na/q}}
		} - 3 \abs{\omega} N \\ &
		\geq \frac{1}{2\abs{q\omega}} - 3 \abs{\omega} N 
		= \frac{1}{\abs{q}} \parentheses*{
			\frac{1}{2\abs{\omega}}
			- 3C \abs{\omega} \frac{\abs{n}}{\abs{q}}
		}.
	\end{align*}
	Now, by assumption, the term in the parentheses is $\geq(4\abs{\omega})^{-1}$, and the assertion of the lemma follows.
\end{proof}

\subsection{Proof of \texorpdfstring{\cref{thm:ExpSumBound}}{Theorem\autoref{thm:ExpSumBound}}} \label{ProofofTh4.4.}
Assume the hypotheses of \cref{thm:ExpSumBound} and let $n$ and $\tilde{n}$ be two distinct algebraic integers in $\integers$ which coincide modulo $q$. Then there is some non-zero $m\in\integers$ such that $n-\tilde{n}=mq$ and, hence, $\abs{n-\tilde{n}}=\abs{m}\abs{q}\geq\abs{q}$.
Assuming $a$ and $q$ to be coprime and $\abs{n-\tilde{n}}<\abs{q}$, we conclude that $(n-\tilde{n})a$ is divisible by $q$ if and only if $n=\tilde{n}$. Consequently, if $\mathscr{R}\subseteq\CC$ is some set with
\begin{equation}\label{eq:SquareDiameter}
	\diam\mathscr{R} \leq \frac{\abs{q}}{12C\abs{\omega}^2},
\end{equation}
then, according to \cref{lem:FractionalPartBounds:Pertubed}, any two distinct points $n\alpha$, $\tilde{n}\alpha$ ($n,\tilde{n} \in \mathscr{X}\cap\mathscr{R}$) satisfy the spacing condition
\[
	\max\braces{
		\norm{\ReO{\omega}\parentheses{(n-\tilde{n})\alpha}},
		\norm{\ImO{\omega}\parentheses{(n-\tilde{n})\alpha}}
	}
	\geq \frac{1}{4\abs{q\omega}}.
\]
Therefore, for $0<\Delta_1 ,\Delta_2\leq\frac{1}{2}$, the sum
\[
	\sum_{\substack{
		n\in\mathscr{X}\cap\mathscr{R} \\
		\norm{\ReO{\omega}(n\alpha)} \leq \Delta_1 \\
		\norm{\ImO{\omega}(n\alpha)} \leq \Delta_2
	}} \!\! 1
\]
is bounded by
\[
	\sum_{\substack{
		n\in\mathscr{X}\cap\mathscr{R} \\
		\braces{\ReO{\omega}(n\alpha)} \leq \Delta_1 \\
		\braces{\ImO{\omega}(n\alpha)} \leq \Delta_2
	}} \!\! 1
	+ \sum_{\substack{
		n\in\mathscr{X}\cap\mathscr{R} \\
		\braces{\ReO{\omega}(n\alpha)} \geq 1-\Delta_1 \\
		\braces{\ImO{\omega}(n\alpha)} \leq \Delta_2
	}} \!\! 1
	+ \sum_{\substack{
		n\in\mathscr{X}\cap\mathscr{R} \\
		\braces{\ReO{\omega}(n\alpha)} \leq \Delta_1 \\
		\braces{\ImO{\omega}(n\alpha)} \geq 1-\Delta_2
	}} \!\! 1
	+ \sum_{\substack{
		n\in\mathscr{X}\cap\mathscr{R} \\
		\braces{\ReO{\omega}(n\alpha)} \geq 1-\Delta_1 \\
		\braces{\ImO{\omega}(n\alpha)} \geq 1-\Delta_2
	}} \!\! 1
\]
which in turn
is bounded by four times the maximum number of points of pairwise maximum norm distance $\geq(4\abs{\omega}\abs{q})^{-1}$ that can be put in a rectangle with side lengths $\Delta_1$ and $\Delta_2$, i.e.,
\begin{equation}\label{eq:CountingBound}
	\sum_{\substack{
		n\in\mathscr{X}\cap\mathscr{R} \\
		\norm{\ReO{\omega}(n\alpha)} \leq \Delta_1 \\
		\norm{\ImO{\omega}(n\alpha)} \leq \Delta_2
	}} \!\!\!\! 1
	\ll \parentheses{1+\abs{q\omega}\Delta_1 } \parentheses{1+\abs{q\omega}\Delta_2}.
\end{equation}

Moving on, let $L\in\NN$ be a parameter at our disposal. Then the
sum
\[
	S(\mathscr{R}) = \sum_{n\in\mathscr{X}\cap\mathscr{R}} E(n,M)
\]
with $E(n,M)$ as defined in~\cref{eq:LinearSumBound} admits a decomposition
\begin{align*}
	S(\mathscr{R}) &
	\leq \sum_{\substack{
		n\in\mathscr{X}\cap\mathscr{R} \\
		\norm{\ReO{\omega}(n\alpha)} \leq 2^{-L} \\
		\norm{\ImO{\omega}(n\alpha)} \leq 2^{-L}
	}} \!\!\!\! M
	\begin{multlined}[t]
		+ \mathop{\sum\!\sum\!\sum}_{\substack{
			2\leq k_1,k_2 \leq L,\, n\in\mathscr{X}\cap\mathscr{R} \\ \\
			2^{-k_1} < \norm{\ReO{\omega}(n\alpha)} \leq 2^{1-k_1} \\
			2^{-k_2} < \norm{\ImO{\omega}(n\alpha)} \leq 2^{1-k_2}
		}} \!\!\!\! \min\braces{ 2^{k_1}, 2^{k_2} } +{} \\
		+ \sum_{2\leq k\leq L} \braces[\bigg]{
			\hspace{-1em}
			\sum_{\substack{
				n\in\mathscr{X}\cap\mathscr{R} \\
				2^{-k} < \norm{\ReO{\omega}(n\alpha)} \leq 2^{1-k} \\
				\norm{\ImO{\omega}(n\alpha)} \leq 2^{-L}
			}}
			\hspace{-1ex} + \hspace{-1ex}
			\sum_{\substack{
				n\in\mathscr{X}\cap\mathscr{R} \\
				\norm{\ReO{\omega}(n\alpha)} \leq 2^{-L} \\
				2^{-k} < \norm{\ImO{\omega}(n\alpha)} \leq 2^{1-k}
			}}
			\hspace{-1em}
		} 2^k \hfill
	\end{multlined} \\ &
	= S_1(\mathscr{R}) + S_2(\mathscr{R}) + S_3(\mathscr{R}),
	\quad\text{say}.
\end{align*}
By~\cref{eq:CountingBound} and using $(a+b)^2\leq 2a^2+2b^2$ ($a,b\geq 1$),
\[
	S_1(\mathscr{R})
	\ll \parentheses{ 1 + \abs{q\omega} 2^{-L} }^2 M
	\ll M + \abs{q\omega}^2 2^{-2L} M.
\]
Moreover, using
\(
	\min\braces{ 2^{k_1}, 2^{k_2} } \leq \sqrt{ 2^{k_1+k_2} }
\)
and~\cref{eq:CountingBound},
\[
	S_2(\mathscr{R})
	\ll \parentheses[\bigg]{
		\sum_{2\leq k\leq L} 2^{k/2} \parentheses{ 1 + \abs{q\omega} 2^{-k} }
	}^2
	\ll 2^{L} + \abs{q\omega}^2.
\]
Similarly,
\[
	S_3(\mathscr{R})
	\ll \sum_{2\leq k\leq L} 2^k \parentheses{1+\abs{q\omega}2^{-k} } \parentheses{1+\abs{q\omega}2^{-L}}
	\ll 2^L + \abs{q\omega}^2 L.
\]

Assuming $M \geq 2$, we take $L = \ceil{\frac{1}{2\log2} \log(2M)}$ to obtain
\begin{equation}\label{eq:S:R:Bound:1}
	S(\mathscr{R})
	\ll M + \abs{q\omega}^2 \log(2M)
	\ll M + \Norm(q\omega) \log M.
\end{equation}
Additionally, if
\(
	x \leq \abs{q}^2 /\parentheses{12C\abs{\omega}^2 }^2
\),
then we take $L = \ceil{  \frac{1}{\log2}\log\parentheses{4\abs{q\omega}} }$. In this case \cref{lem:FractionalPartBounds:Pertubed} shows that $S_1(\mathscr{R})$ vanishes and, consequently, we have
\begin{equation}\label{eq:S:R:Bound:2}
	S(\mathscr{R})
	\ll \abs{q\omega}^2 \log\parentheses{2\abs{q\omega}}
	\ll \Norm(q\omega) \log \Norm(q\omega).
\end{equation}
Finally, we note that the set $\mathscr{X}$ can be covered by fewer than
\begin{equation}\label{eq:NumberOfRects:Bound}
	\parentheses[\bigg]{ 1 + \frac{2\sqrt{x}}{\braces{\text{diameter~bound}}/\sqrt{2}} }^2
	\ll 1 + C^2 \abs{\omega}^4 x \abs{q}^{-2}
\end{equation}
squares $\mathscr{R}$ with diameter~\cref{eq:SquareDiameter}.
Together with~\cref{eq:S:R:Bound:1} this proves~\cref{eq:ExpSumAverageBound:1},
and together with~\cref{eq:S:R:Bound:2} we obtain~\cref{eq:ExpSumAverageBound:2}.
This proves \cref{thm:ExpSumBound}.

\begin{proof}[Proof of \cref{thm:Diophantine}]
	The assertion concerning the vanishing of $H_\alpha(x,\Delta)$ is contained in \cref{lem:FractionalPartBounds:Pertubed}. As for the bound for $H_\alpha(x,\Delta)$, cover the set $\mathscr{X} = \braces{ n\in\integers : 0<\Norm(n) \leq x }$ with rectangles as above and employ~\cref{eq:CountingBound}.
\end{proof}

\section{The non-smoothed version}
\label{sec:HarmanInput}

Here we tackle the problem of verifying the assumptions of \cref{thm:HarmansSieve} in a setting suitable for proving \cref{thm:DiophApprox:Integers}. Throughout, we assume the hypotheses of \cref{thm:DiophApprox:Integers}, although $x\geq 3$ may be considered arbitrary until \cref{subsec:AssemblingTheParts}, where we take $x=\Norm(q)^{28/5}$.

\subsection{Setting up linear and bilinear forms}
\label{subsec:SettingUp}

Let $\mathscr{B} = \braces{ n\in\integers : x/2 \leq \Norm(n) < x }$ and $\mathscr{A} = \braces{ n\in\mathscr{B} : \norm{ n\alpha }_\omega < \delta }$.
Concerning \cref{thm:HarmansSieve}, we choose $\tilde{w}$ to be $\boldsymbol{1}_{\mathscr{A}}$, the characteristic function of the set $\mathscr{A}$, and $w = 4\delta^2 \boldsymbol{1}_{\mathscr{B}}$.
Given these definitions, the limit in \eqref{eq:Convergence} is actually attained for every $R\geq x$ and trivial estimates suffice to show that $X$ therein may be taken $\ll x^5$.\footnote{%
	Of course, for such divisor sums much better estimates are available (see, e.g., \cref{lem:CombinedDivisorBound} below for $d_4$ replaced with $d_2$).
	However, since in \cref{thm:HarmansSieve}, only the logarithm of $X$ enters in the final error term, we can be very sloppy here.
}
Moving on, we shall want to compare sums of the type
\begin{equation}\label{eq:General:TypeI/II:Sum}
	\mathop{\sum\!\sum}_{mn\in\mathscr{A}}a_{m}b_{n}
	\quad\text{and}\quad
	4\delta^2 \mathop{\sum\!\sum}_{mn\in\mathscr{B}}a_{m}b_{n},
\end{equation}
where the summation indices $m,n$ vary through $\integers$ and the coefficient sequences $(a_{m})_{m}$ and $(b_{n})_{n}$ consist of complex numbers and satisfy $\abs{a_{n}}\leq 1$ and $\abs{b_{n}}\leq d(n\integers)$.\\
To be more specific, for parameters $\mu>0$ and $0<\kappa\leq\frac{1}{2}$,
there are two types of sums we would like to estimate
\begin{itemize}
	\item \emph{Type~I}: $b_{n}=1$ in the above and $(a_{m})_{m}$ is supported only on $m$ with $0<\Norm(m)<M$ for some $M$ with $x^{\mu}<M\leq x$ (see~\cref{eq:Harman:Type:I}).
	\item \emph{Type~II}: $(a_{m})_{m}$ is supported only on $m$ with $x^{\mu}\leq \Norm(m)<x^{\mu+\kappa}$
	(see~\cref{eq:Harman:Type:II}).
\end{itemize}
Each type requires a different treatment, but for now it is convenient
to start by transforming~\cref{eq:General:TypeI/II:Sum} without
restricting to either of the above types. We start with the following
result which furnishes a finite Fourier approximation to the \emph{saw-tooth function} $\psi$ given by
\[
	\psi(t) = t - \floor{t} - \tfrac{1}{2}
	\quad (t\in\RR).
\]
\begin{lem}
	\label{lem:Fourier}
	For all real $x$ and  $J\ge 1$, we have
	\[
		\psi(x) = \sum\limits_{1\le \abs{j}< J} (2\pi i j)^{-1} \e{-jx} + O\parentheses*{\min\braces*{ \log 2J, \frac{1}{J\norm{x}} }}.
	\]
\end{lem}
\begin{proof}
	This is Lemma~4.1.2 in~\cite{Brudern1995ANT}. 
\end{proof}

We now derive a useful expansion of the characteristic function $\boldsymbol{1}_{\mathscr{A}}$ of $\mathscr{A}$ evaluated at algebraic integers.
For an element $y$ of $\integers$ we put
\[
	x_{1,y} = \ReO{\omega}(y\alpha)
	\quad\text{and}\quad
	x_{2,y} = \ImO{\omega}(y\alpha).
\]
Furthermore, let
\[
	\integers_{\delta,\alpha} = \braces{ y\in\integers : x_{1,y} \in \pm\delta+\ZZ \text{ or } x_{2,y} \in \pm\delta+\ZZ }.
\]
We now consider the characteristic function $\boldsymbol{1}_{\mathscr{A}}\colon \integers\to\braces{0,1}$ of the set $\mathscr{A}$.
For any $y\in\mathscr{B}\setminus\integers_{\delta,\alpha}$ we have the expansion
\begin{equation}\label{eq:BeattyIndicatorFuncDecomp}
	\newcommand{\charExpansion}{\parentheses{
		\psi(- x_{k,y} - \delta)
		- \psi(- x_{k,y} + \delta)
	}}
	\begin{aligned}
		\boldsymbol{1}_{\mathscr{A}}(y) &
		= \prod_{k=1,2} \parentheses{ 2\delta + \charExpansion } \\ &
		\multlinegap=0pt
		= 4\delta^2 \begin{multlined}[t]
			+ 2\delta \sum_{k=1,2} \charExpansion +{} \\
			\shoveleft[0pt]{ + \prod_{k=1,2} \charExpansion \hfill }
		\end{multlined} \\ &
		= 4\delta^2 + 2\delta \sum_{k=1,2} \Xi_k(y) + \Xi_3(y),
		\quad\text{say}.
	\end{aligned}
\end{equation}
Note here that the first equality in~\cref{eq:BeattyIndicatorFuncDecomp} may not hold for $y\in\integers_{\delta,\alpha}$.
Nevertheless, the last line in~\cref{eq:BeattyIndicatorFuncDecomp} remains bounded even in that case.
Therefore one can, as we do below, also use the last line of~\cref{eq:BeattyIndicatorFuncDecomp} as a substitute for $\boldsymbol{1}_{\mathscr{A}}(y)$ even when $y\in\integers_{\delta,\alpha}$.
This only introduces an error bounded by a constant times for how many $y\in\integers_{\delta,\alpha}$ is this applied.

For $k=1,2,3$ we consider the sums
\[
	\varSigma_k = \MNsum a_mb_n \Xi_k(mn).
\]
For $k=1,2$, on applying \cref{lem:Fourier} with some $J\ge 1$
to be specified later (see~\cref{eq:ParameterChoices} below), for \emph{any} choice of summation ranges for $m,n$, we have
\begin{align*}
	\varSigma_k 
	= & 
	\MNsum a_mb_n \sum\limits_{1\le \abs{j}< J} (2\pi i j)^{-1}
	\parentheses{\e{-j\delta}-\e{j\delta}} \e{j(-x_{k,mn})} + O(G) \\
	\ll & \sum\limits_{1\le \abs{j}< J}\Pi(j) \cdot \abs[\bigg]{ \MNsum a_mb_n \e{ jx_{k,mn} } } + G,
\end{align*}
where
\begin{equation}\label{eq:Weights}
	\Pi(j) = \min\braces{ \abs{j}^{-1}, \delta }
\end{equation}
and 
\begin{equation} \label{Gdef}
	G \coloneqq \sum\limits_{k=1}^2 \sum\limits_{l=0}^1 \MNsum \abs{a_m b_n}\min\braces*{\log 2J, 
	\frac{1}{J\norm{(-1)^l\delta-x_{k,mn}}}}.
\end{equation}

Similarly, we obtain 
\[
	\varSigma_3 \ll 
		\sum\limits_{\substack{1\le \abs{j_1}< J\\ 1\le \abs{j_2}< J}} 
		\Pi(j_1)\Pi(j_2)\cdot \abs[\bigg]{ \MNsum a_mb_n \e{-j_1x_{1,mn}-j_2x_{2,mn} } } +  (\log 2J)G,
\]
where we have used the trivial estimates
\[
	\sum\limits_{1\le \abs{j}< J} (2\pi ij)^{-1} \parentheses{\e{-j\delta}-\e{j\delta}} \e{j(-x_{k,mn})}\ll \log 2J
\]
and 
\[
	\min\braces*{\log 2J, \frac{1}{J\norm{(-1)^l\delta-x_{k,mn}}}} \ll \log 2J.
\]

Now consider
\begin{equation}\label{eq:BilinearFormError}
	\bilinError
	= \mrestrictedsum{mn\in\mathscr{A}} a_mb_n
	- 4\delta^2 \mrestrictedsum{mn\in\mathscr{B}} a_mb_n,
\end{equation}
where the star in the summation indicates that the range of $m$ is to be restricted to a Type~I or Type~II range.
The first sum may be written as
\[
	\mrestrictedsum{mn\in\mathscr{A}} a_mb_n = \mrestrictedsum{mn\in\mathscr{A}} a_mb_n \boldsymbol{1}_{\mathscr{A}}(mn)
\]
and we can apply~\cref{eq:BeattyIndicatorFuncDecomp} to all those terms where $mn\notin\integers_{\delta,\alpha}$.
On the other hand, we may use the last line of~\cref{eq:BeattyIndicatorFuncDecomp} as a substitute for $\boldsymbol{1}_{\mathscr{A}}(mn)$ for \emph{all} terms $mn$ in the above at the cost of an error $O(G)$ (see the comment just below~\cref{eq:BeattyIndicatorFuncDecomp}).
Then, combining this with our analysis of the sums $\varSigma_k$ from above, we find that
\begin{equation}\label{eq:BilinearFormError:Estimate}
	\multlinegap=0cm
	E \ll
	\begin{multlined}[t]
		(\log 2J)G + \delta \max_{k=1,2} \sum_{1\leq\abs{j}<J} \Pi(j) \abs[\bigg]{ \mrestrictedsum{mn\in\mathscr{B}} a_mb_n \e{-jx_{k,mn}} } +{} \\
		\shoveleft[.2cm]{ + \JJsum \Pi(j_1)\Pi(j_2) \abs[\bigg]{ \mrestrictedsum{mn\in\mathscr{B}} a_mb_n \e{ -j_1x_{1,mn} -j_2x_{2,mn} } }. }
	\end{multlined}
\end{equation}

\subsection{Removing the weights: dyadic intervals}

Here we shall remove the weights~\cref{eq:Weights} attached to
the sums in~\cref{eq:BilinearFormError:Estimate}. This
may be achieved by splitting the summation over $j$ (or $j_1 $,
$j_2$) into dyadic intervals: indeed, for any non-negative $f:\ZZ^2 \to\RR$,
letting
\begin{equation}\label{eq:F}
	F(J_1,J_2) = \mathop{\sum\!\sum}_{\substack{
		0\leq\abs{j_1}<J_1 \\
		0\leq\abs{j_2}<J_2 \\
		\mathclap{ (j_1,j_2) \neq (0,0) }
	}} f(j_1,j_2),
\end{equation}
we find that
\begin{gather*}
	\JJsum \Pi(j_1)\Pi(j_2) f(j_1,j_2)
	\ll (\log J)^2 \max_{\substack{1\leq J_1 \leq J \\ 1\leq J_2\leq J}} \Pi(j_1)\Pi(j_2) F(J_1,J_2), \\
	\sum_{1\leq\abs{j}<J} \Pi(j) f(j,0)
	\ll (\log J) \max_{1\leq J_1 \leq J} \Pi(J_1) F(J_1,1), \\
	\sum_{1\leq\abs{j}<J} \Pi(j) f(0,j)
	\ll (\log J) \max_{1\leq J_2 \leq J} \Pi(J_2) F(1,J_2).
\end{gather*}
Of course, we shall apply this with
\begin{equation}\label{eq:f}
	f(j_1,j_2)
	= \abs[\bigg]{ \mrestrictedsum{mn\in\mathscr{B}} a_mb_n \e{-j_1 x_{1,mn} -j_2x_{2,mn}} }.
\end{equation}

Now assume for the moment that we have bounds
\begin{equation}\label{eq:FJ1J2:Bound}
	F(J_1,J_2) \ll \mathcal{F}(J_1,J_2),
\end{equation}
where the right-hand side is symmetric in both arguments and does not depend on the particular choice of the coefficients in~\cref{eq:f} (but, of course, still subject to the Type~I/II conditions presented in \cref{subsec:SettingUp});
the reader may wish to glance at \cref{prop:TypeII:ExplicitBound} and \cref{prop:TypeI:ExplicitBound} below, 
where we furnish such bounds for the Type~II and Type~I sums respectively. 

Then, using~\cref{eq:BilinearFormError:Estimate}  we have
\begin{equation}\label{eq:BilinearFormError:Estimate:Simplified}
	\frac{\abs{\bilinError}}{J^\epsilon}
	\ll_{\epsilon}
	\max_{1\leq J_1 \leq J} \delta\Pi(J_1) \mathcal{F}(J_1,1)
		+ \max_{\substack{ 1\leq J_1\leq J \\ 1\leq J_2\leq J }} \Pi(J_1)\Pi(J_2) \mathcal{F}(J_1,J_2) + G.
\end{equation}
We shall return to this in \cref{subsec:AssemblingTheParts} and now focus on establishing the aforementioned bounds of the shape \cref{eq:FJ1J2:Bound}.

\subsection{Transforming the argument in the exponential term}
\label{subsec:TransformingArg}
In the proof of the bounds for the Type~I and Type~II sums we need to combine variables in $\integers$ (see \cref{subsec:Bound:TypeII,subsec:Bound:TypeI} below).
Having this goal in mind, the shape of the argument of the exponential in~\cref{eq:f} appears to be, at a superficial glance, a technical obstruction.

However, this putative problem vanishes after a simple variable transformation that we shall now describe: by definition of $x_{k,mn}$,
\begin{equation}\label{eq:ExponentialArgument}
	- j_1 x_{1,mn} - j_2x_{2,mn}
	= - j_1 \ReO{\omega}(mn\alpha) - j_2\ImO{\omega}(mn\alpha).
\end{equation}
Letting $\xi_2$ be given as in~\cref{eq:OmegaSquared:XiDecomp} and writing $\ell = \ell_1 + \ell_2\omega$, a short computation yields
\begin{equation} \label{Imcomp}
	\ImO{\omega}(\ell\rho)
	= \ell_2\ReO{\omega}\rho + (\ell_1 +\ell_2\xi_2) \ImO{\omega}\rho
	\quad(\rho\in\CC).
\end{equation}
Then, via the equivalence
\[
	\begin{pmatrix} 0 & -1 \\ -1 & -\xi_2 \end{pmatrix}
	\begin{pmatrix} \ell_1 \\ \ell_2 \end{pmatrix}
	= \begin{pmatrix} j_1 \\ j_2 \end{pmatrix}
	\Longleftrightarrow
	\begin{pmatrix} \ell_1 \\ \ell_2 \end{pmatrix}
	= \begin{pmatrix} \xi_2 & -1 \\ -1 & 0 \end{pmatrix}
	\begin{pmatrix} j_1 \\ j_2 \end{pmatrix} \!,
\]
and assuming $(j_1,j_2)\neq(0,0)$, we observe that~\cref{eq:ExponentialArgument} equals $\ImO{\omega}\parentheses{\ell mn\alpha}$ when $(\ell_1,\ell_2)$ is calculated via the above formula.
We let $\mathscr{L}(J_1,J_2)$ be the set of algebraic integers $\ell\in\integers$ arising from $(j_1,j_2)$ via the above formula, that is, $\mathscr{L}\parentheses{J_1,J_2}$ is the set
\[
	\braces{ (\xi_2j_1 -j_2) - j_1\omega :
		\abs{j_1 }<J_1,\,
		\abs{j_2}<J_2,\,
		(j_1 ,j_2)\neq(0,0)
	}.
\]
Consequently, if $F(J_1 ,J_2)$ is given by~\cref{eq:F} with $f$ given by~\cref{eq:f}, then
\begin{align}\label{eq:F:J1:J2}
	F(J_1,J_2)
	= \sum_{\ell\in\mathscr{L}(J_1 ,J_2)} \abs[\bigg]{ \mrestrictedsum{mn\in\mathscr{B}} a_mb_n \e{\ImO{\omega}\parentheses{ \ell mn\alpha }} }.
\end{align}

For a later extension of the summation over $\ell$, we note that, using \cref{lem:AlgebraicFacts}~\cref{enu:ImOmegaLowerBound}, $\mathscr{L}(J_1,J_2)$ can be seen to be contained in the set of all $\ell$ satisfying
\begin{align}\label{eq:ell:NormBound}
	1 \leq \Norm(\ell) < 5 \Norm(\omega^2) \parentheses{J_1+J_2}^2.
\end{align}
The reader will note that this set potentially contains many more elements than $\mathscr{L}\parentheses{J_1 ,J_2}$, for we obviously have
\begin{equation}\label{eq:LH1H2:CountBound}
	\#\mathscr{L}\parentheses{J_1,J_2}
	\leq \parentheses{2J_1+1} \parentheses{2J_2+1}
	\leq 9\, J_1 J_2;
\end{equation}
In any case, we require both~\cref{eq:ell:NormBound} and~\cref{eq:LH1H2:CountBound}.

\subsection{The Type~II sums}\label{subsec:Bound:TypeII}
In this section, we establish the following.
\begin{prop}[Type~II bound]
	\label{prop:TypeII:ExplicitBound}
	Consider $F$ from~\cref{eq:F}
	with $f$ given by~\cref{eq:f} subject to
	\[
		\mrestrictedsum{mn\in\mathscr{B}}
		= \mathop{\sum\!\sum}_{\substack{
			x/2\leq \Norm(mn)<x \\
			\mathclap{ x^{\mu} < \Norm(m) < x^{\mu+\kappa} }
		}
	},
	\]

	where $\mu\in(0,1]$, $\kappa\in(0,\frac{1}{2}]$ and $x\geq 3$. For the coefficients in~\cref{eq:f} assume that $\abs{a_m}\leq 1$ and $\abs{b_{n}}\leq d(n\integers)$.
	Moreover, suppose that $a$, $q$, $\gamma$ and $C$ are as in~\cref{eq:ThetaApprox}.
	Then, for any $\epsilon\in(0,\frac{1}{2}]$,
	\begin{equation}\label{eq:TypeII:ExplicitBound}
		F(J_1,J_2) \ll_{\epsilon} \!\begin{multlined}[t]
			C \Norm(\omega)^{7/2} x^{\epsilon}
			\bigl\lparen
			J_1J_2 x^{(1+\mu+\kappa)/2} +
			(J_1J_2)^{1/2+\epsilon} \times{} \\
			\shoveleft[.2cm]{
				\times \bigl\lparen
				(J_1+J_2) x \Norm(q)^{-1/2}
				+ (J_1+J_2) x^{1-\mu/4}
				+{}
			} \\
			+ \Norm(q)^{1/2} x^{(2+\mu+\kappa)/4}
				\bigr\rparen.
		\end{multlined}
	\end{equation}
\end{prop}

In the course of the proof of \cref{prop:TypeII:ExplicitBound} and at other places we need the following lemma to control trivial sums over $m$ and $n$.
\begin{lem}
	\label{lem:CombinedDivisorBound}
	Let $\numberfieldK$ be a fixed quadratic number field and $\integers$ its ring of integers. For an ideal $\mathfrak{a}\subseteq\integers$ let $d(\mathfrak{a})$ denote the number of ideals $\mathfrak{b}\supseteq\mathfrak{a}$, and fix $\epsilon>0$ and some integer $\ell\geq 2$. Then, for $x\geq 2$,
	\begin{enumerate}
		\item\label{enu:DivisorAverageBounds} \(\displaystyle
		\sum_{N\mathfrak{a}\leq x} d(\mathfrak{a})^\ell \ll_{\integers,\ell} x \parentheses{\log x}^{2^{2\ell}-1}
		\),
		\item\label{enu:DivisorBound} $d(\mathfrak{a}) \ll_{\integers,\epsilon} (N\mathfrak{a})^{\epsilon}$,
	\end{enumerate}
	where the implied constants depend at most on $\integers$, $\ell$ and $\epsilon$.
\end{lem}
\begin{proof}
	The first assertion is a direct consequence of \cite{lu2011on-mean-values}. On the other hand, the second assertion is immediate from the first.
\end{proof}

Using \cref{lem:AlgebraicFacts}~\cref{enu:UnitsInRingOfIntegers} and \cref{lem:CombinedDivisorBound}~\cref{enu:DivisorAverageBounds}, we have
\begin{equation*}
	\mathop{\sum\!\sum}_{mn\in\mathscr{B}} 1
	\leq \#\braces{\text{units in }\integers} \cdot \sum_{N\mathfrak{a}<x} d(\mathfrak{a})
	\ll x (\log x)^3.
\end{equation*}
(Note that here the dependence on $\integers$ in \cref{lem:CombinedDivisorBound} can be neglected, as we are dealing only with the finitely many imaginary quadratic number fields $\numberfieldK$ with class number~$1$.)

\begin{proof}[Proof of \cref{prop:TypeII:ExplicitBound}]
	Looking at~\cref{eq:F:J1:J2}, we may split the summation over $m$ into \enquote{dyadic annuli,} getting
	\begin{equation}\label{eq:F(H1H2):Decomposition:TypeII}
		F(J_1,J_2)
		\ll (\log x) \max_{\substack{
			x^{\mu} < K,K' \leq x^{\mu+\kappa} \\
			K \leq K'< 2K
		}
		} F(J_1 ,J_2,K,K'),
	\end{equation}
	where, upon employing the transformation described in \cref{subsec:TransformingArg} along the way, $\hat{F} = F(J_1 ,J_2,K,K')$ may be taken to be
	\[
		\mathop{\sum\!\sum}_{\substack{ \ell\in\mathscr{L}(J_1,J_2) \\ K\leq \Norm(m) < K' }} \abs[\bigg]{
			\sum_{x/2\leq \Norm(nm)<x} b_n \e{\ImO{\omega}(\ell mn\alpha)}
		}.
	\]
	(Here and in the following we are always assuming $J_1$, $J_2$, $K$ and $K'$ to be positive integers such that $K\leq K'<2K$.)
	By~\cref{eq:LH1H2:CountBound} and \cref{lem:AlgebraicFacts}~\cref{enu:LatticeCount:UpperBound},
	\[
		\mathop{\sum\!\sum}_{\substack{ \ell\in\mathscr{L}(J_1,J_2) \\ \mathclap{ K\leq \Norm(m) < K' } }} 1
		\ll J_1 J_2 K.
	\]
	Hence, letting
	\begin{equation}\label{eq:Q:Definition}
		Q = (\hat{F})^2 / (J_1 J_2 K),
	\end{equation}
	Cauchy's inequality gives
	\[
		Q \leq \mathop{\sum\!\sum}_{\substack{ \ell\in\mathscr{L}(J_1,J_2) \\ K\leq \Norm(m) < K' }} \abs[\bigg]{
			\sum_{x/2\leq \Norm(nm)<x} b_n \e{\ImO{\omega}(\ell mn\alpha)}
		}^2 ,
	\]
	which, upon expanding the square and rearranging, yields
	\[
		Q \leq \mathop{\sum\!\sum\!\sum}_{\substack{
			\ell \in \mathscr{L}(J_1,J_2) \\
			\mathclap{ x/(2K') \leq \Norm(n) < x/K } \\
			\mathclap{ x/(2K') \leq \Norm(\tilde{n}) < x/K }
		}} \abs{b_n\overline{b_{\tilde{n}}}} \abs[\bigg]{
			{\sideset{}{^{*}}\sum_{m}} \e{\ImO{\omega}\parentheses{\ell m\parentheses{n-\tilde{n}}\alpha}}
		},
	\]
	where ${\sideset{}{_{m}^{*}}\sum}$ restricts the summation to those $m$ with
	\[
		\max\braces{ K, x/2\Norm(n), x/2\Norm(\tilde{n}) }
		\leq \Norm(m)
		< \min\braces{ K', x/\Norm(n), x/\Norm(\tilde{n}) }.
	\]

	Next, we isolate the \enquote{diagonal contribution} $\Delta$, that is, those terms where $n=\tilde{n}$, for in this case the sum over $m$ can only be bounded trivially. Using \cref{lem:CombinedDivisorBound}~\cref{enu:DivisorAverageBounds},~\cref{eq:LH1H2:CountBound} and \cref{lem:AlgebraicFacts}~\cref{enu:LatticeCount:UpperBound}, this is found to be
	\begin{equation}\label{eq:DiagonalBound}
		\begin{aligned}
			\Delta &
			= \sum_{x/(2K') \leq \Norm(n) < x/K} \abs{b_n}^2 \! \sum_{\ell\in\mathscr{L}(J_1,J_2)} \!  {\sideset{}{^{*}}\sum_{m}} 1 \\ &
			\ll x K^{-1} (\log x)^{15} J_1 J_2 K \\ &
			\ll J_1 J_2 x (\log x)^{15}.
		\end{aligned}
	\end{equation}
	Moreover, using~\cref{eq:ell:NormBound}, \cref{lem:CombinedDivisorBound}, \cref{enu:DivisorBound}, \cref{lem:LinearExpSumBound}
	and~\cref{eq:ell:NormBound}, we have
	\[
		Q \ll_{\epsilon} \Delta + \Norm(\omega) \parentheses{x/K}^{2\epsilon} \sqrt{K'}
		\sum_{1 \leq \Norm(j) < \newUpperBound} c_j E(j,\sqrt{K'}),
	\]
	where
	\begin{gather}\label{eq:newUpperBound}
		\newUpperBound = 20 \Norm(\omega^2) \parentheses{J_1 + J_2}^2 x K^{-1}, \\ \nonumber
		c_j = \sum_{\substack{ \ell\in\mathscr{L}(J_1 ,J_2) \\ \ell\mid j }} \mathop{\sum\!\sum}_{\substack{
				x/(2K') \leq \Norm(n)<x/K \\
				x/(2K') \leq \Norm(\tilde{n}) < x/K \\
				j/\ell = (n-\tilde{n})
		}} 1
		\ll_{\epsilon} \newUpperBound^{\epsilon} \cdot x/K
	\end{gather}
	and $E$ is given by~\cref{eq:LinearSumBound}.
	Thus, using~\cref{eq:DiagonalBound} and \cref{thm:ExpSumBound}, and recalling~\cref{eq:Q:Definition},
	\[
		(\hat{F})^2 \ll_{\epsilon} \begin{multlined}[t]
			(J_1J_2)^2 x K (\log x)^{15} + \Norm(\omega) (x/K)^{2\epsilon} \newUpperBound^{\epsilon} J_1 J_2 x \sqrt{K} \times{} \\
			\shoveleft[.5cm]{ \times \parentheses{ 1 + C^2 \Norm(\omega^2) \newUpperBound / \Norm(q) }
			\parentheses{ \sqrt{K} + \Norm(q\omega) \log\parentheses{2\sqrt{K}} }. \hfill }
		\end{multlined}
	\]
	Upon taking the square root, and simplifying the resulting expressions,
	\[
		\hat{F} \ll_{\epsilon} \begin{multlined}[t]
			C \Norm(\omega)^{7/2} x^{2\epsilon} \bigl\lparen J_1J_2 \sqrt{xK} + (J_1J_2)^{1/2+\epsilon} \times{} \\
			\shoveleft[.5cm]{ \times \parentheses{ (J_1+J_2) x \Norm(q)^{-1/2} + (J_1+J_2) x K^{-1/4} + \Norm(q)^{1/2} x^{1/2} K^{1/4} } \bigr\rparen. \hfill }
		\end{multlined}
	\]
	Recalling~\cref{eq:F(H1H2):Decomposition:TypeII}, we infer~\cref{eq:TypeII:ExplicitBound} after adjusting $\epsilon$.
\end{proof}

\subsection{The Type~I sums}\label{subsec:Bound:TypeI}
The next step is to estimate the Type I sums. We establish the following.

\begin{prop}[Type~I bound]
	\label{prop:TypeI:ExplicitBound}
	Consider $F$ from~\cref{eq:F} with $f$ given by~\cref{eq:f} subject to
	\[
		\mrestrictedsum{mn\in\mathscr{B}}
		= \mathop{\sum\!\sum}_{\substack{ x/2 \leq \Norm(mn) < x \\ \Norm(m)<M }},
	\]
	where $M\leq x$ and $x\geq3$. For the coefficients in~\cref{eq:f} assume that $\abs{a_m}\leq 1$ and $b_n = \boldsymbol{1}_{\braces{1 \leq \Norm(n) < x}}$.
	Moreover, suppose that $a$, $q$, $\gamma$ and $C$ are as in~\cref{eq:ThetaApprox}.
	Then, for any $\epsilon \in \lparen 0,\frac{1}{2} \rbrack$,
	\begin{equation}\label{eq:TypeI:ExplicitBound}
		F(J_1,J_2) \ll_\epsilon \begin{multlined}[t]
			C^2 \Norm(\omega)^7 (J_1+J_2)^{\epsilon} (x \Norm(q))^{\epsilon} \times{} \\
			\shoveleft[.5cm]{ \times \parentheses[\big]{ (J_1+J_2)^2 x / \Norm(q) + (J_1+J_2)^2 x^{1/2} M^{1/2} + x^{1/2} \Norm(q) }. \hfill }
		\end{multlined}
	\end{equation}
\end{prop}
\begin{proof}
	As we did with the Type~II sums in the proof of \cref{prop:TypeII:ExplicitBound}, we may split the summation over $m$ into dyadic annuli, getting
	\begin{equation}\label{eq:F(H1H2):Decomposition:TypeI}
		F(J_1,J_2)
		\ll (\log x) \max_{\substack{
			1 \leq K,K' \leq M \\
			K \leq K'< 2K
		}} \tilde{F}(J_1,J_2,K,K'),
	\end{equation}
	where $\tilde{F} = \tilde{F}(J_1 ,J_2,K,K')$ is given by
	\[
		\mathop{\sum\!\sum}_{\substack{ \ell\in\mathscr{L}(J_1,J_2) \\ K\leq \Norm(m)<K' }} \abs[\bigg]{
			\sum_{x/2\leq \Norm(nm) < x} \e{ \ImO{\omega}(\ell mn\alpha) }
		}.
	\]
	Letting $R=\sqrt{x/K}$ and employing \cref{lem:LinearExpSumBound} as well as \cref{lem:CombinedDivisorBound}~\cref{enu:DivisorBound}, we infer
	\begin{align*}
		\tilde{F} &
		\ll \Norm(\omega) R \mathop{\sum\!\sum}_{\substack{ \ell\in\mathscr{L}(J_1,J_2) \\ \mathclap{ K\leq \Norm(m)<K' } }} E(\ell m,R) \\ &
		\ll_{\epsilon} \Norm(\omega) \newUpperBound^{\epsilon} R \!\!\!\! \sum_{K\leq \Norm(j)<\newUpperBound} \!\!\!\! E(j,R),
	\end{align*}
	with $E$ given by~\cref{eq:LinearSumBound} and $U$ by~\cref{eq:newUpperBound} with $K'$ in place of $x/K$.
	\cref{thm:ExpSumBound} now shows that
	\begin{align*}
		\tilde{F} &
		\ll_{\epsilon} \newUpperBound^{\epsilon} \Norm(\omega)  R \parentheses{ 1 + C^2 \Norm(\omega^2) \newUpperBound / \Norm(q) } \parentheses{ R + \Norm(q\omega) \log\parentheses{ 2R } } \\ &
		\ll_{\epsilon} \!\begin{multlined}[t]
			\Norm(\omega^2) (J_1+J_2)^{2\epsilon} x^{\epsilon} (x K^{-1}) + C^2 \Norm(\omega^7) (J_1+J_2)^{2\epsilon} x^{\epsilon} \times{} \\
			\shoveleft[.5cm]{ \times \parentheses{ (J_1+J_2)^2 x / \Norm(q) + (J_1+J_2)^2 x^{1/2} K^{1/2} + x^{1/2} \Norm(q) K^{-1/2} }. \hfill }
		\end{multlined}
	\end{align*}
	Herein, for very small $K$, the term $xK^{-1}$ becomes problematic.
	To circumvent this, we note that \cref{thm:ExpSumBound} also furnishes the bound
	\[
		\tilde{F}
		\ll_{\epsilon} \Norm(\omega)^{7/2} (J_1+J_2)^{2\epsilon} x^{1/2+\epsilon} \Norm(q)^{1+\epsilon} K^{-1/2},
	\]
	provided that
	\begin{equation}\label{eq:newUpperBound:SizeCaseOne}
		\newUpperBound \leq \Norm(q) / \parentheses{ 12C \Norm(\omega) }^2 .
	\end{equation}
	On the other hand, if~\cref{eq:newUpperBound:SizeCaseOne} fails to hold, then, recalling~\cref{eq:newUpperBound}, we have
	\[
		x K^{-1}
		= 10 \Norm(\omega^2) (J_1 +J_2)^2 x\newUpperBound^{-1}
		\ll \Norm(\omega^4) C^2 (J_1+J_2)^2 x / \Norm(q).
	\]
	Therefore, after joining both bounds,
	\[
		\tilde{F}
		\ll_{\epsilon} \!\begin{multlined}[t]
			C^2 \Norm(\omega)^7 (J_1+J_2)^{2\epsilon} (x \Norm(q))^\epsilon \times{} \\
			\shoveleft[.5cm]{ \times \parentheses{ (J_1+J_2)^2 x \Norm(q)^{-1} + (J_1+J_2)^2 x^{1/2} K^{1/2} + 
			x^{1/2} \Norm(q) K^{-1/2} }. }
		\end{multlined}
	\]
	Upon plugging this into~\cref{eq:F(H1H2):Decomposition:TypeI}, we obtain~\cref{eq:TypeI:ExplicitBound} after adjusting $\epsilon$.
\end{proof}

\subsection{Estimation of~\texorpdfstring{$G$}{G}}
The final task is to bound the error term $G$, defined in~\cref{Gdef}. We shall establish the following.

\begin{prop}
	Consider $G$ from~\cref{Gdef}, and suppose that $a$, $q$, $\gamma$ and $C$ are as in~\cref{eq:ThetaApprox}.
	Then we have
	\begin{equation}\label{G12est}
		G
		\ll_{\epsilon} C^2 \Norm(\omega)^3 (xJ)^{\epsilon}
		\parentheses*{
			\Norm(q)^{1/2} +
			\frac{\Norm(q)}{J} +
			\frac{x}{\Norm(q)^{1/2}} +
			\frac{x}{J}
		}.
	\end{equation}
\end{prop}

\begin{proof}
Using the definition of $x_{k,mn}$, writing $r=mn$ and using \cref{lem:CombinedDivisorBound}, we obtain
\[
	G
	\ll_\epsilon \!\begin{multlined}[t]
	(xJ)^{\epsilon}\sum\limits_{l=0}^1 \sum\limits_{1\le \Norm(r)\le x}
	\biggl\lparen \min\braces*{ 1, \frac{1}{J\norm{(-1)^l\delta-\Im_{\omega}(r\theta)}} } +{} \\
	+ \min\braces*{ 1, \frac{1}{J\norm{(-1)^l\delta-\Re_{\omega}(r\theta)}} } \biggr\rparen.
	\end{multlined}
\]
We shall bound
\[
	\sum\limits_{1\le \Norm(r)\le x} m_r,
	\quad\text{where}\quad
	m_r = \min\braces*{ 1, \frac{1}{J\norm{\delta-\Im_{\omega}(r\theta)}} }
\]
and treat the remaining three sums of this type similarly. To this end, similarly as in \cref{ProofofTh4.4.}, 
we cover the set of $r$'s in question, 
\[
	\mathscr{X}=\braces{ r\in \integers \ :\ 1\le \Norm(r)\le x },
\]
by $O(1 + C^2 N(\omega)^2 x / N(q))$ many rectangles $\mathscr{R}$ with diameter satisfying~\cref{eq:SquareDiameter} (see also~\cref{eq:NumberOfRects:Bound}),
so that
\[
	\mathscr{X}\subset \integers \cap \bigcup\limits_{\mathscr{R}} \mathscr{R}.
\]
Furthermore, we write
\[
	\sum\limits_{r\in \mathscr{X} \cap \mathscr{R}} m_r
	\ll \sum\limits_{j=0}^J
	\min\braces*{ 1,\frac{1}{J\norm{j/J}} } \sum\limits_{\substack{r\in \mathscr{X} \cap \mathscr{R}\\ j/J\le \braces{\delta-\Im_{\omega}(r\theta)}\le (j+1)/J}} 1.
\]
Similarly as in \cref{ProofofTh4.4.} (see~\cref{eq:CountingBound}), we establish that
\[
	\sum\limits_{\substack{r\in \mathscr{X} \cap \mathscr{R}\\ j/J\le \{\delta-\Im_{\omega}(r\theta)\}\le (j+1)/J}} 1
\ll \Norm(q\omega)^{1/2} \parentheses*{ 1+\frac{\Norm(q\omega)^{1/2}}{J} }.
\]
It follows that
\begin{align*}
	\sum\limits_{1\le \Norm(r)\le x} m_r &
	\ll_{\epsilon} \parentheses*{ 1 + \frac{C^2 N(\omega)^2 x}{N(q)} } J^{\epsilon}\Norm(q\omega)^{1/2} \parentheses*{ 1+\frac{\Norm(q\omega)^{1/2}}{J} } \\ &
	= J^{\epsilon}
	\parentheses*{
		\Norm(q\omega)^{1/2} +
		\frac{\Norm(q\omega)}{J} +
		\frac{C^2 \Norm(\omega)^{5/2} x}{\Norm(q)^{1/2}} +
		\frac{C^2 \Norm(\omega)^3 x}{J}
	}.
\end{align*}
Treating the remaining three sums of this type similarly, we obtain \cref{G12est} after adjusting $\epsilon$. 
\end{proof}

\subsection{Assembling the parts}
\label{subsec:AssemblingTheParts}

Finally, we are in a position to use~\cref{eq:BilinearFormError:Estimate:Simplified}.
Assume the hypotheses of \cref{prop:TypeII:ExplicitBound}. Recall~\cref{eq:Weights}. Set
\[
	\tilde{G}\coloneqq \Norm(q)^{1/2}+\Norm(q)J^{-1}+x\Norm(q)^{-1/2}+xJ^{-1},
\]
the term in the brackets on the right-hand side of \cref{G12est}.
Then, looking at~\cref{eq:TypeII:ExplicitBound}, we use~\cref{eq:BilinearFormError:Estimate:Simplified}
and \cref{G12est} together with the inequalities
\begin{gather*}
	\Pi(H) \cdot H \ll 1, \quad \Pi(H) \cdot H^{3/2}\le J^{1/2}, \quad
	\Pi(H) \cdot H^{1/2}\le \delta^{1/2}, \\
	\Pi(H_1)\Pi(H_2) \cdot H_1H_2 \ll 1, \quad
	\Pi(H_1)\Pi(H_2)\cdot (H_1H_2)^{1/2}(H_1+H_2) \ll (\delta J)^{1/2}
\end{gather*}
and
\[
	\Pi(H_1)\Pi(H_2) \cdot (H_1H_2)^{1/2} \ll \delta
\]
if $H,H_1,H_2\le J$,
to bound the error in the Type~II sums (see \cref{eq:BilinearFormError}) as
\[
	\frac{\abs{\bilinError}}{ C^2 \Norm(\omega)^{7/2} (xJ)^\epsilon }
	\ll_{\epsilon} x^{(1+\mu+\kappa)/2} \begin{multlined}[t]
		+ (\delta J)^{1/2} x \Norm(q)^{-1/2} +{} \\
		+ (\delta J)^{1/2} x^{1-\mu/4} + \delta \Norm(q)^{1/2} x^{(2+\mu+\kappa)/4}+\tilde{G}, \hfill
	\end{multlined}
\]
$\epsilon$ being sufficiently small.

Moving on to the Type~I sums, accordingly assuming the hypotheses of \cref{prop:TypeI:ExplicitBound}
and looking at~\cref{eq:TypeI:ExplicitBound}, we use~\cref{eq:BilinearFormError:Estimate:Simplified} and \cref{G12est} 
together with the inequalities
\begin{gather*}
	\Pi(H)\cdot H^2 \ll J, \quad \Pi(H) \ll \delta, \\
	\Pi(H_1)\Pi(H_2)\cdot \parentheses{H_1^2+H_2^2} \ll \delta J
	\quad \mbox{and} \quad
	\Pi(H_1)\Pi(H_2) \ll \delta^2
\end{gather*}
to infer the estimate
\[
	\frac{\abs{\bilinError}}{C^2 \Norm(\omega)^7 (xJ)^\epsilon}
	\ll_{\epsilon} \Norm(q)^\epsilon \parentheses[\big]{ \delta J x \Norm(q)^{-1} + \delta J x^{1/2} M^{1/2} + \delta^2 x^{1/2} \Norm(q) }
	+\tilde{G}
\]
for the error in the Type I sums.

On recalling~\cref{eq:BilinearFormError} and plugging
the above bounds into \cref{thm:HarmansSieve}, we find that
the error
\[
	\tilde{\bilinError} = \frac{
		\abs{ S(w,x^{\kappa}) - S(\tilde{w},x^{\kappa}) }
	}{
		C^2 \Norm(\omega)^7
	}
\]
satisfies the bound
\[
	\frac{\tilde{\bilinError}}{(xJ)^\epsilon}
	\ll_{\epsilon} \! \begin{multlined}[t]
		\parentheses[\big]{\Norm(q)^{1/2}+\Norm(q)J^{-1}+x\Norm(q)^{-1/2}+xJ^{-1}} +{} \\
		\shoveleft[.5cm]{ + \Norm(q)^{\epsilon} \bigl\lparen \delta J x \Norm(q)^{-1} + \delta J x^{1/2} M^{1/2} +{} } \\
		\hfill + \delta^2 x^{3/4} \Norm(q)^{1/2} + \delta^2 x^{1/2} \Norm(q) \bigr\rparen +{} \\
		\shoveleft[.5cm]{ + x^{\parentheses{1+\mu+\kappa}/2} + (\delta J)^{1/2} x \Norm(q)^{-1/2} +{} } \\
		\shoveleft[.5cm]{ + (\delta J)^{1/2} x^{1-\mu/4} + \delta \Norm(q)^{1/2} x^{1/2+(\mu+\kappa)/4}. }
	\end{multlined}
\]
Evidently, this bound is increasing with $\kappa$ and to detect primes, we must take $\kappa = \frac{1}{2}$.
In view of~\cref{eq:NumberFieldPNT}, we shall aim for a bound of the type
\begin{equation}\label{eq:BilinError:TargetBound}
	\abs{\tilde{\bilinError}}
	\ll_{\epsilon} (\delta^2 x) x^{-\epsilon}
\end{equation}
with $\delta$ in some range (w.r.t.\ $x$) as large as possible. 
With this constraint in mind, and given $q$, we take $x = \Norm(q)^{28/5}$ (as was stated in \cref{thm:DiophApprox:Integers}) 
so that $\Norm(q) = x^{5/28}$ and, moreover,
\begin{equation}\label{eq:ParameterChoices}
	J = \ceil{ \delta^{-2} x^{2\epsilon} },\quad
	M = x^{1/2},\quad
	\mu = \tfrac{5}{14}.
\end{equation}
Then, under the additional assumption that $J\le x$, we obtain
\[
	\tilde{\bilinError} \ll_{\epsilon} \delta^2 x^{1-\epsilon} + 
x^{5\epsilon} \parentheses*{ \delta x^{45/56}+x^{13/14}+ \delta^{-1/2}x^{51/56}+\delta^{-1}x^{23/28} },
\]
provided $\epsilon$ is sufficiently small. This implies that \cref{eq:BilinError:TargetBound} holds for sufficiently small $\epsilon$
and 
\[
	\delta\ge x^{-1/28+3\epsilon},
\]
which concludes the proof of \cref{thm:DiophApprox:Integers} after adjusting $\epsilon$. 

\section{The smoothed version}
\label{sec:Smoothed}

Here, in a similar vein to \cref{sec:HarmanInput}, we work on providing the details for what was outlined in \cref{rem:OverallProcedure}. However, this time the aim is to prove \cref{thm:ApproxWithPrimes}.

\subsection{The modified setup}
\label{sec:HypothesesForSmoothedVersion}

Throughout the rest of \cref{sec:Smoothed} we make the following assumptions: $\epsilon>0$ is supposed to be sufficiently small and fixed. $\numberfieldK$ is an imaginary quadratic number field with class number~$1$. The number $x\geq 3$ is assumed to be sufficiently large and $\alpha,a,q,C$ are as in \cref{thm:DiophApprox:Integers}.
Moreover, we suppose that
\begin{equation}\label{eq:DeltaBounds}
	\tfrac{1}{2} \geq \delta \geq x^{-1000}.
\end{equation}
The exact lower bound here is of no particular consequence, as our final results even fall short of being non-trivial for $\delta \leq x^{-1/16}$. However, in the course of getting there, we need to have bounds of the shape $\delta^{-1}e^{-x^\epsilon} \ll_A x^{-A}$ for any $A>0$ as $x\to\infty$. Such bounds are used---often tacitly---throughout.\\
Furthermore, we write
\begin{equation}\label{eq:N:and:fN:def}
	N = x^{1-\epsilon}
	\quad\text{and}\quad
	f_{N}(z) = e^{-\pi\abs{z}^{2}/N}
\end{equation}
and define the weight function $w$ to be used in conjunction with \cref{thm:HarmansSieve} by
\[
	w(z)=\delta^{2}f_{N}(z).
\]
To define $\tilde{w}$, we let
\[
	W_{\delta}(\vartheta)
	= \sum_{n\in\ZZ} e^{-\pi(\vartheta-n)^{2}/\delta^{2}}
\]
which by Poisson summation formula implies
\[
W_{\delta}(\vartheta)= \delta\sum_{j\in\ZZ} e^{-\pi\delta^{2}j^{2}} \e{j\vartheta}.
\]
Then let
\[
	\tilde{w}(z) = f_{N}(z)W_{\delta} \parentheses{\ImO{\omega}(z\alpha)} W_{\delta}\parentheses{\ReO{\omega}(z\alpha)+\xi_{2}\ImO{\omega}(z\alpha)}
\]
with $\xi_2$ from~\cref{eq:OmegaSquared:XiDecomp}.

\subsection{Removing the weights}
\label{sec:Removing-the-weights}

Our next immediate goal is to see that $w$ and $\tilde{w}$ are actually suitable weights for the type of argument outlined in \cref{rem:OverallProcedure}.
This is contained in \cref{cor:Detector} below, but first we need two lemmas. We use the notation $S(\,\cdot\,,\sqrt{x})$ from~\cref{eq:SiftingFunction}.

\begin{lem}
	\(\displaystyle
		S(w,\sqrt{x}) \gg_{\epsilon} \delta^{2} \frac{N}{\log N}
	\).
\end{lem}
\begin{proof}
	The claim follows at once from
	\[
		S(w,\sqrt{x})
		= \sum_{\substack{r\in\integers\setminus\braces{0} \\ \text{prime }p \mid r \Rightarrow \Norm(p)\geq\sqrt{x}}} w(r)
		\geq \sum_{\substack{\text{prime }p\in\integers \\ \sqrt{x}\leq\Norm(p)\leq N }} w(p)
		\geq \delta^{2} e^{-\pi} \sum_{\substack{ \text{prime }p\in\integers \\ \sqrt{x}\leq\Norm(p)\leq N }} 1
	\]
	after applying the prime number theorem for $\integers$ (see~\cref{eq:NumberFieldPNT}).
\end{proof}

\begin{lem}
	\(\displaystyle
		S(\tilde{w},\sqrt{x})
		\ll_{\epsilon,\omega} \sum_{\substack{
			\text{\rm prime }p\in\integers \\
			\Norm(p)<x \\
			\norm{p\alpha}_{\omega}<\delta x^{\epsilon}
		}} e^{-\pi\Norm(p)/N}
		+ 1
	\).
\end{lem}
\begin{proof}
	We may split up $S(\tilde{w},\sqrt{x})$ as follows:
	\[
		\braces[\Big]{
			\sum_{\substack{
				\text{prime }p\in\integers \\
				\sqrt{x}\leq\Norm(p)<x \\
				\norm{p\alpha}_{\omega}<\delta x^{\epsilon}
			}}
			+ \sum_{\substack{
				\text{prime }p\in\integers \\
				\sqrt{x}\leq\Norm(p)<x \\
				\norm{p\alpha}_{\omega} \geq \delta x^{\epsilon}
			}
		}}
		\braces{\text{terms}}
		+ \braces{\text{error term}},
	\]
	where the terms being summed are
	\[
		W_{\delta}\parentheses{\ImO{\omega}(p\alpha)} W_{\delta}\parentheses{\ReO{\omega}(p\alpha)+\xi_{2}\ImO{\omega}(p\alpha)}
		e^{-\pi\Norm(p)/N}
	\]
	and
	\[
		\braces{\text{error term}}
		\ll_{\epsilon} 1 + \sum_{\substack{ r\in\integers\\ \Norm(r)\geq x }} e^{-\pi\Norm(r)/N}
		\ll_{\epsilon} 1.
	\]
	The first sum is bounded using the trivial estimate $W_{\delta}(\,\cdot\,)\ll 1$. \\
	Now if $\norm{p\alpha}_\omega \geq \delta x^{\epsilon}$, then $\norm{\ImO{\omega}(p\alpha)}\geq\delta x^{\epsilon}/2\xi_{2}$, or $\norm{\ImO{\omega}(p\alpha)} \leq \delta x^{\epsilon}/2\xi_{2}$ \emph{and} $\norm{\ReO{\omega}(p\alpha)}\geq\delta x^{\epsilon}$.
	In the first case, using the inequality
	\begin{equation}\label{eq:squareineq}
		(\norm{\vartheta}-m)^{2}
		\geq m^{2}/4
		\quad (\vartheta\in\RR, \, m\in\ZZ\setminus\braces{0}),
	\end{equation}
	we have
	\begin{align*}
		W_{\delta}\parentheses{\ImO{\omega}(p\alpha)} &
		= \sum_{m\in\ZZ}e^{-\pi(\norm{\ImO{\omega}(p\alpha)}-m)^{2}/\delta^{2}}
		\ll e^{-\pi x^{2\epsilon}/(2\xi_{2})^{2}} + \sum_{m=1}^{\infty} e^{-\frac{\pi}{4}m^{2}/\delta^{2}} \\ &
		\ll_{A,\epsilon,\omega} \delta^{2}x^{-A-1},
	\end{align*}
	where for the last estimate we employ~\cref{eq:DeltaBounds}. In the second case, the assumptions ensure that $\norm{\ReO{\omega}(p\alpha)+\xi_{2}\ImO{\omega}(p\alpha)}\geq\delta x^{\epsilon}/2$.
	Therefore, by arguing as before, we have
	\[
		W_{\delta}\parentheses{\ReO{\omega}(p\alpha)+\xi_{2}\ImO{\omega}(p\alpha)}
		\ll e^{-\frac{\pi}{4}x^{2\epsilon}} + \sum_{m=1}^{\infty} e^{-\frac{\pi}{4}m^{2}/\delta^{2}}
		\ll_{A,\epsilon} \delta^{2}x^{-A-1}.
	\]
	Thus, altogether we have
	\begin{align*}
		\sum_{\substack{
				\text{prime }p\in\integers \\
				\Norm(p)<x \\
				\norm{p\alpha}_{\omega}\geq\delta x^{\epsilon}
		}}
		\begin{multlined}[t]
			W_{\delta}\parentheses{\ImO{\omega}(p\alpha)}
			W_{\delta}\parentheses{\ReO{\omega}(p\alpha)+\xi_{2}\ImO{\omega}(p\alpha)} e^{-\pi\Norm(p)/N} \\[2mm]
			\ll_{A,\epsilon,\omega} \delta^{2} x^{-A-1} \sum_{\substack{ r\in\integers \\ \mathclap{\Norm(r)<x} }} 1
			\ll_{A,\epsilon,\omega} \delta^{2} x^{-A}.
		\end{multlined}
	\end{align*}
	This proves the lemma.
\end{proof}
\begin{cor}
	\label{cor:Detector}
	Still assuming the hypotheses from \cref{sec:HypothesesForSmoothedVersion}, suppose that one knows that
	\[
		\abs{ S(w,\sqrt{x}) - S(\tilde{w},\sqrt{x}) }
		= o_{\epsilon,\omega}\parentheses*{ \delta^{2} \frac{N}{\log N} },
	\]
	where $S(\,\cdot\,,\sqrt{x})$ is defined as in~\cref{eq:SiftingFunction}.
	Then
	\[
		\sum_{\substack{
			\text{\rm prime }p\in\integers \\
			\Norm(p)<x \\
			\norm{p\alpha}_{\omega}<\delta x^{\epsilon}
		}} e^{-\pi\Norm(p)/N}
		\gg_{\epsilon,\omega} \delta^{2} \frac{N}{\log N}.
	\]
	In particular, for any sufficiently large $x$, there is a prime element $p\in\integers$ such that $\Norm(p)<x$ and $\norm{p\alpha}_{\omega} < \delta x^{\epsilon}$.
\end{cor}

\subsection{Estimation of smoothed sums}
The next result is a smoothed analogue of \cref{lem:LinearExpSumBound}. Later, this is used in combination with \cref{thm:Diophantine}. (The reader may contrast this with our use of \cref{thm:ExpSumBound} as the underlying tool for proving \cref{prop:TypeI:ExplicitBound} and \cref{prop:TypeII:ExplicitBound}.)
\begin{lem}
	\label{lem:PoissonResult-1}
	Let $R \geq 1$ and $\gamma\in\RR$. Then, for every $\epsilon>0$ and $f_R$ defined as in~\cref{eq:N:and:fN:def},
	\[
		\abs[\Big]{ \sum_{m\in\integers} f_{R}(m) \e{\ImO{\omega}(m\vartheta)} }
		\ll_{\epsilon,\omega}
		\begin{cases}
			R & \text{if } \norm{\vartheta}_{\omega} < x^{\epsilon}/\sqrt{R} \\
			R e^{-x^{\epsilon}} & \text{otherwise}.
		\end{cases}
	\]
\end{lem}
\begin{proof}
	Let $A_\omega$ be the invertible $(2\times 2)$-matrix
	\(
		\parentheses*{\!\begin{smallmatrix}
			1 & \Re\omega \\ 0 & \Im\omega
		\end{smallmatrix}\!}
	\)
	and write $A_\omega^{-\top} = (A_\omega^{-1})^{\top}$.
	Moreover, recall that $\omega^{2}=\xi_{1}+\xi_{2}\omega$ and let $v_{\vartheta} = ( \ImO{\omega}\vartheta, \xi_{2}\ImO{\omega}\vartheta + \ReO{\omega}\vartheta )$.
	Then, writing $g(u,v) = e^{-\pi(u^{2}+v^{2})}$, we have
	\[
		\Sigma
		\coloneqq \sum_{m\in\integers} f_{R}(m) \e{\ImO{\omega}(m\vartheta)}
		= \sum_{m\in\ZZ^{2}} g(R^{-1/2}A_\omega m) \e{\inner{m,v_{\vartheta}}}.
	\]
	By the Poisson summation formula (see, e.g., \cite{stein1971intro-fourier-analysis}) and a change of variables,
	\begin{align*}
		\Sigma &
		= \sum_{n\in\ZZ^{2}} \int_{\RR^{2}} g(R^{-1/2}A_\omega m) \e{\inner{m,v_{\vartheta}}} \e{-\inner{m,n}} \dd{m} \\ &
		= \frac{R}{\abs{\det A_\omega}} \sum_{n\in\ZZ^{2}} \int_{\RR^{2}} g(y) \e{-\inner{R^{1/2}A_\omega^{-1}y,n-v_{\vartheta}}} \dd{y} \\ &
		= \frac{R}{\abs{\det A_\omega}} \sum_{n\in\ZZ^{2}} \int_{\RR^{2}} g(y) \e{-\inner{ y,R^{1/2}A_{\omega}^{-\top}(n-v_{\vartheta})}} \dd{y}\!.
	\end{align*}
	Using the fact that $g$ is its own Fourier transform, we have
	\[
		\Sigma = \frac{R}{\abs{\Im\omega}} \sum_{ m \in A^{-\top} (\ZZ^{2} - v_{\vartheta}) } g(R^{1/2}m).
	\]
	A quick computation shows that
	\[
		A_\omega^{-\top} \parentheses*{ \! \begin{pmatrix} n_{1} \\ n_{2} \end{pmatrix} - v_{\vartheta} \! }
		= \begin{pmatrix}
			n_{1} - \ImO{\omega} \vartheta \\
			\frac{ n_{2} - \ReO{\omega}\vartheta - (2\Re\omega)\ImO{\omega}\vartheta - (n_{1} - \ImO{\omega}\vartheta) \Re\omega }{\Im\omega}
		\end{pmatrix}.
	\]
	Consequently, after a linear change of summation variables,
	\[
		\Sigma = \frac{R}{\abs{\Im\omega}} \mathop{\sum\sum}_{k_{1},k_{2}\in\ZZ} e^{-\pi R\parentheses{ (k_{1}-\kappa_{1})^{2} + (k_{2}-\kappa_{2}(k_{1}))^{2} / (\Im\omega)^{2} }},
	\]
	where
	\begin{gather*}
		\kappa_{1} = \norm{\ImO{\omega}\vartheta}, \\
		\kappa_{2}(k_{1}) = \norm{ \ReO{\omega}\vartheta - (2\Re\omega)(\ImO{\omega}\vartheta) - ( k_{1} - \norm{\ImO{\omega}\vartheta} ) \Re\omega }.
	\end{gather*}
	As the bound $\Sigma \ll R$ is trivial, we now assume that
	\[
		\tfrac{1}{2} \geq \max\braces{ \norm{\ReO{\omega}\vartheta}, \norm{\ImO{\omega}\vartheta} } = \norm{\vartheta}_{\omega} \geq x^{\epsilon}/\sqrt{R}.
	\]
	In particular, we have $R \geq 4 x^{2\epsilon}$. To bound $\Sigma$ we split off the term for $k_{1} = k_{2} = 0$, namely
	\[
		\Sigma_{0} = \frac{R}{\abs{\Im\omega}} e^{ -\pi R\parentheses{ \kappa_{1}^{2} + \kappa_{2}(0)^{2} / (\Im\omega)^{2} } },
	\]
	from the rest, that is,
	\[
		\Sigma = \Sigma_0 + \Sigma_*, \quad
		\Sigma_{*} = \frac{R}{\abs{\Im\omega}} \mathop{\sum\sum}_{\substack{
			k_{1},k_{2}\in\ZZ \\
			\mathclap{(k_{1},k_{2})\neq(0,0)}
		}} e^{-\pi R\parentheses{ (k_{1}-\kappa_{1})^{2} + (k_{2}-\kappa_{2}(k_{1}))^{2} / (\Im\omega)^{2} }}.
	\]
	Using~\cref{eq:squareineq}, we have
	\begin{align*}
		\Sigma_{*} &
		\ll R \mathop{\sum\sum}_{\substack{
			k_{1},k_{2}\in\ZZ \\
			\mathclap{(k_{1},k_{2})\neq(0,0)}
		}} e^{-\pi R\parentheses{ k_{1}^{2}+k_{2}^{2} / (\Im\omega)^{2} }/4}
		\ll R \sum_{k=1}^{\infty} \mathop{\sum\sum}_{\substack{
			k_{1},k_{2}\in\ZZ \\
			k_{1}^{2}+k_{2}^{2} = k
		}} e^{-\pi Rk / \max\braces{ 4, 4(\Im\omega)^{2} }} \\ &
		\ll e^{-\pi R/\max\braces{ 8, 8(\Im\omega)^{2} }}
		\ll_{\epsilon,\omega} e^{-x^{\epsilon}}.
	\end{align*}
	To bound $\Sigma_{0}$, we put $r = \max\braces{ 1, 2\abs{\Re\omega} }$.
	By assumption, $\norm{\vartheta}_{\omega} \geq x^{\epsilon}/\sqrt{R}$, so that we either have $\norm{\ImO{\omega}\vartheta} \geq x^{\epsilon} / (r\sqrt{R})$, or $\norm{\ImO{\omega}\vartheta} < x^{\epsilon} / (r\sqrt{R})$ \emph{and} $\norm{\ReO{\omega}\vartheta} \geq x^{\epsilon} / \sqrt{R}$. In the former case,
	\[
		\Sigma_{0}
		\ll R e^{-\pi R\kappa_{1}^{2}}
		\ll R e^{-\pi x^{2\epsilon}/r^{2}}
		\ll_{\epsilon,\omega} R e^{-x^{\epsilon}}.
	\]
	In the latter case, we first note that $2\Re\omega$ is an integer (see \cref{lem:AlgebraicFacts}~\cref{enu:2ReOmegaIsIntegral}),
	so that
	\[
		\kappa_{2}(0)
		= \norm{ \ReO{\omega}\vartheta - \parentheses{(2\Re\omega)(\ImO{\omega}\vartheta) - \norm{\ImO{\omega}\vartheta} \Re\omega} }
		= \norm{ \ReO{\omega}\vartheta - \parentheses{\norm{\ImO{\omega}\vartheta} \Re\omega} }
	\]
	and, hence,
	\[
		\Sigma_{0}
		\ll R e^{-\pi R\kappa_{2}(0)^{2}/(\Im\omega)^{2}}
		\ll R e^{-\pi x^{2\epsilon}/(2\Im\omega)^{2}}
		\ll_{\epsilon} R e^{-x^{\epsilon}}.
	\]
	Recalling that the original sum under consideration, $\Sigma$, is $\Sigma_{0} + \Sigma_{*}$, the assertion of the lemma follows.
\end{proof}

\subsection{Cutting off}
\begin{lem}
	\label{lem:CuttingOff}
	Consider the sum
	\[
		\Sigma = \sideset{}{^{*}}\sum_{m\in\integers} \sum_{n\in\integers} a_{m} b_{n} (w(mn)-\tilde{w}(mn)),
	\]
	where $a_{m}$ and $b_{n}$ are arbitrary complex coefficients satisfying $\abs{a_{m}} \leq 1$, $a_{0} = b_{0} = 0$, $\abs{b_{n}} \leq \#\braces{\text{divisors of }n}$, and $\sum_{m}^{*}$ means some arbitrary restriction on the summation over $m\in\integers$. Then, for every $A, \epsilon > 0$,
	\[
		\abs{\Sigma} \ll_{A,\epsilon,\omega} \delta^{2} \sideset{}{^{*}}\sum_{\substack{
			j\in\integers\setminus\braces{0} \\
			\Norm(j)\leq\delta^{-2}x^{\epsilon}}} 
		 \abs[\Big]{\sum_{m\in\integers} 		
			\sum_{n\in\integers} a_mb_{n} f_{N}(mn) \e{ \ImO{\omega} \parentheses{jmn\alpha} } }
		+ \delta^{2} x^{-A}.
	\]
\end{lem}
\begin{proof}
	Recalling~\cref{Imcomp}, we have
	\begin{align*}
		\tilde{w}(z) - w(z) &
		= f_{N}(z) \delta^{2} \mathop{\sum\sum}_{\substack{
			j_{1},j_{2}\in\ZZ \\
			\mathclap{ (j_{1},j_{2}) \neq (0,0) }
		}}
		\begin{multlined}[t]
			e^{-\pi\delta^{2}(j_{1}^{2}+j_{2}^{2})} \times{} \\
			\quad \times \e{ j_{1}\ImO{\omega}(z\alpha) + j_{2}\parentheses{ \ReO{\omega}(z\alpha) + \xi_{2}\ImO{\omega}(z\alpha) } }
		\end{multlined} \\ &
		= f_{N}(z) \delta^{2} \mathop{\sum\sum}_{\substack{
			j_{1},j_{2}\in\ZZ \\
			\mathclap{ (j_{1},j_{2}) \neq (0,0) }
		}} e^{-\pi\delta^{2}(j_{1}^{2}+j_{2}^{2})} \e{\ImO{\omega}\parentheses{ (j_{1}+j_{2}\omega)z\alpha }}.
	\end{align*}
	Consequently,
	\begin{equation}\label{eq:cutoff}
		\abs{\Sigma}
		\leq \delta^{2} \mathop{\sum\sum}_{\substack{
			j_{1},j_{2}\in\ZZ \\
			\mathclap{ (j_{1},j_{2})\neq(0,0) } \\
			\mathclap{ j\coloneqq j_{1}+j_{2}\omega }
		}} e^{-\pi\delta^{2}(j_{1}^{2}+j_{2}^{2})} \abs[\Big]{ \sideset{}{^{*}}\sum_{m\in\integers} a_m 
			\sum_{n\in\integers} b_{n} f_{N}(mn) \e{ \ImO{\omega}\parentheses{jmn\alpha} }
		}.
	\end{equation}
	The inner-most sum over $n$ is bounded by $e^{-\pi\Norm(m)/N}$ multiplied by
	\[
		\sum_{n\in\integers} \abs{b_{n}} e^{-\pi\Norm(n)/N}
		\ll \sum_{r=1}^{\infty} e^{-\pi r/N} \sum_{\substack{ n\in\integers \\ \Norm(n)=r }} \abs{b_{n}}
		\ll \sum_{r=1}^{\infty} e^{-\pi r/N} r^{2}
		\ll N^{3}.
	\]
	Thus,
	\[
		\abs[\Big]{\sideset{}{^{*}}\sum_{m\in\integers} 
			\sum_{n\in\integers} a_mb_{n} f_{N}(mn) \e{\ImO{\omega}\parentheses{jmn\alpha}}
		}
		\ll N^{5}
		\ll x^{5}.
	\]
	A short computation shows that
	\[
		\Norm(j_{1}+j_{2}\omega) \leq \delta^{-2}x^{\epsilon}
		\implies
		\max\braces{\abs{j_{1}},\abs{j_{2}}} \leq c \delta^{-1} x^{\epsilon/2},
	\]
	where $c = (1+\abs{\Re\omega}) / \abs{\Im\omega}$. Consequently,
	\begin{align*}
		\smash{\mathop{\sum\sum}_{\substack{
			j_{1},j_{2}\in\ZZ \\
			\max\braces{\abs{j_{1}},\abs{j_{2}}} > c\delta^{-1}x^{\epsilon/2}
		}}} e^{-\pi\delta^{2}(j_{1}^{2}+j_{2}^{2})} &
		\ll \sum_{\substack{ r=1 \\ r>c^{2}\delta^{-2}x^{\epsilon} }}^{\infty} r e^{-\pi\delta^{2}r}
		\ll \delta^{-4} x (1+c^{2}) e^{-\pi c^{2} x^{2\epsilon}} \\ &
		\ll_{A,\epsilon,\omega} \delta^{2} x^{-A-5}.
	\end{align*}
	Hence, we obtain the claimed result from~\cref{eq:cutoff} after estimating the contribution from all terms with $\Norm(j) > \delta^{-2} x^{\epsilon}$ via the above and using the trivial inequality $e^{-\pi\delta^{2} (j_{1}^{2} + j_{2}^{2})} \leq 1$ on the remaining terms.
\end{proof}

\subsection{Type~I estimates}
\begin{prop}[Type~I estimate]
	\label{prop:Type-I-1}
	Consider the sum $\Sigma$ from \cref{lem:CuttingOff}
	with $b_{n}=1$ for all $n\in\integers\setminus\braces{0}$ and
	\[
		\sideset{}{^{*}}\sum_{m}
		= \sum_{ 0 < \Norm(m) \leq M }
	\]
	for some positive $M$.
	Then
	\[
		\abs{\Sigma}
		\ll_{C,\epsilon,\omega} \delta^{2} N \cdot x^{5\epsilon} \parentheses{
			\abs{q}^{2} x^{-1}
			+ \delta^{-2} \abs{q}^{-2}
			+ \delta^{-2} M x^{-1}
		}.
	\]
\end{prop}
\begin{proof}
	By \cref{lem:CuttingOff} we have
	\[
		\abs{\Sigma}
		\ll_{\epsilon,\omega} \delta^{2} \sum_{\substack{
			j \in \integers\setminus\braces{0} \\
			\Norm(j)\leq\delta^{-2}x^{\epsilon}
		}} \sum_{0<\Norm(m)\leq M} \abs[\Big]{
			\sum_{n\in\integers} f_{N/\Norm(m)}(n) \e{\ImO{\omega}\parentheses{jmn\alpha}}
		} + Mx^{\epsilon}.
	\]
	Therefore, by \cref{lem:PoissonResult-1},
	\[
		\abs{\Sigma}
		\ll_{\epsilon,\omega} \delta^{2} N \mathop{\sum\sum}_{\substack{
			0<\Norm(j) \leq \delta^{-2}x^{\epsilon} \\
			0<\Norm(m) \leq M \\
			\mathclap{ \norm{jm\alpha}_{\omega} < \sqrt{\Norm(m)/x^{1-3\epsilon}} }
		}} (\Norm(m))^{-1} + Mx^{\epsilon}.
	\]
	Moreover, by the second part of \cref{thm:Diophantine}, we find that no terms with $\Norm(m) <  L_{0}$ contribute to the above sum if we set
	\[
		L_0\coloneqq \min\braces*{\frac{x^{1-3\varepsilon}}{(4\abs{q\omega})^2},\frac{\delta^{2}\abs{q}^2}{(12C\abs{\omega}^2)^2} }.
	\]
	Then, upon splitting the summation over $m$ into dyadic annuli, we obtain
	\[
		\abs{\Sigma}
		\ll_{\epsilon,\omega} \delta^{2} N (\log x) \sup_{L_{0}\leq L\leq M} L^{-1} \mathop{\sum\sum}_{\substack{
			0 < \Norm(j) \leq \delta^{-2}x^{\epsilon} \\
			L/2 < \Norm(m) \leq L \\
			\mathclap{ \norm{jm\alpha}_{\omega} < \sqrt{L/x^{1-3\epsilon}} }
		}} 1 + Mx^{\epsilon}.
	\]
	Finally, using \cref{thm:Diophantine} and $N=x^{1-\varepsilon}$,
	\begin{align*}
		\abs{\Sigma} &
		\ll_{\epsilon,\omega} \delta^{2} N x^{\epsilon} \sup_{L_{0} \leq L \leq M} \parentheses{ L^{-1} + \delta^{-2} x^{\epsilon} \abs{q}^{-2} } \parentheses{ 1 + L \abs{q}^{2} x^{-1+3\epsilon} } + Mx^{\epsilon} \\ &
		\ll_{C,\epsilon,\omega} \delta^{2} N x^{5\epsilon} \parentheses{ \abs{q}^{2} x^{-1} + \delta^{-2} \abs{q}^{-2} + \delta^{-2} M x^{-1} }.
		\qedhere
	\end{align*}
\end{proof}

\subsection{Type~II estimates}
\begin{prop}[Type~II estimate]
	\label{prop:Type-II-1}
	Consider the sum $\Sigma$ from \cref{lem:CuttingOff} with
	\[
		\sideset{}{^{*}}\sum_{m}
		= \sum_{x^{\mu} < \Norm(m) \leq x^{\mu+\kappa}},
	\]
	for some $\mu,\kappa\in(0,1)$. Then, for any $\epsilon\in(0,\mu)$,
	\begin{align*}
		\abs{\Sigma} &
		\ll_{C,\epsilon,\omega} \delta^{2} N \cdot x^{7\epsilon} (
			\delta^{-2} \min\{x^{-1/4},x^{(\mu+\kappa-1)/2}\} \\ &
			+ \delta^{-2} \abs{q}^{-1}
			+ \delta^{-1} \abs{q} x^{-1/2}
			+ \delta^{-1} x^{-\mu/2}
			+\delta^{-1}x^{(\mu+\kappa-1)/2}
		).
	\end{align*}
\end{prop}
\begin{proof}
	By \cref{lem:CuttingOff} we have
	\[
		\abs{\Sigma}
		\ll_{A,\epsilon,\omega} \delta^{2} \sum_{\substack{
			j\in\integers\setminus\braces{0} \\
			\Norm(j) \leq \delta^{-2} x^{\epsilon}
		}} \abs[\Big]{ \mathop{\sum\sum}_{\substack{ x^{\mu}<\Norm(m)\leq x^{\mu+\kappa} \\ n\in\integers }} a_mb_{n} f_{N}(mn) \e{\ImO{\omega}\parentheses{jmn\alpha}}
		} + \delta^{2} x^{-A}.
	\]
	Upon splitting the summation over $m$ into dyadic annuli, we obtain
	\begin{align*}
		\abs{\Sigma} \ll_{A,\epsilon,\omega} &
		\delta^{2} x^{-A} +
		\delta^{2} (\log x) \times{} \\ &\times\sup_{ x^{\mu} \leq L \leq x^{\mu+\kappa} } 
	        \sum_{\substack{ j\in\integers\setminus\braces{0} \\ \Norm(j) \leq \delta^{-2} x^{\epsilon} }} \abs[\Big]{ \mathop{\sum\sum}_{\substack{ L<\Norm(m)\leq 2L \\ n\in\integers }} a_m b_n f_{N}(mn) \e{\ImO{\omega}\parentheses{jmn\alpha}}
		}.
	\end{align*}
	By similar arguments as in \cref{sec:Removing-the-weights}, we see that one can restrict the summation over $n$ to 
	$\Norm(n) \leq x/L$ at the cost of an error $\ll_{A,\epsilon,\omega} \delta^{2}x^{-A}$.
	Thus,
	\begin{equation}\label{eq:SigmaDyadic}
		\abs{\Sigma}
		\ll_{A,\epsilon,\omega} \delta^{2} (\log x) \sup_{ x^{\mu} \leq L \leq x^{\mu+\kappa} } \Sigma_{L} + \delta^{2}x^{-A},
	\end{equation}
	where
	\[
		\Sigma_{L}
		=  \sum_{\substack{j\in\integers\setminus\braces{0} \\
			\Norm(j) \leq \delta^{-2} x^{\epsilon}
		}} 
		\abs[\Big]{ \mathop{\sum\sum}_{\substack{ L<\Norm(m)\leq 2L \\ 0<\Norm(n)\leq x/L }} a_mb_{n} f_{N}(mn) \e{\ImO{\omega}\parentheses{jmn\alpha}}
		}.
	\]
	We write
	\[
		\Sigma_{L}
		=  \sum_{\substack{j\in\integers\setminus\braces{0} \\
			\Norm(j) \leq \delta^{-2} x^{\epsilon}
		}} 
		c_j \mathop{\sum\sum}_{\substack{ L<\Norm(m)\leq 2L \\ 0<\Norm(n)\leq x/L }} a_mb_{n} f_{N}(mn) \e{\ImO{\omega}\parentheses{jmn\alpha}}
	\]
	with $c_j$ being suitable complex coefficients satisfying $\abs{c_j}=1$. 
	
	Next, we remove the factor $f_N(mn)$ by writing the Gaussian as an inverse Mellin transform in the form
	\[
		e^{-x^2}=\frac{1}{2\pi i}\cdot {\int\limits_{c-i\infty}^{c+i\infty} x^{-s} \cdot \frac{\Gamma(s/2)}{2} \dd{s}},
	\]
	where $c>0$. This implies
	\[
		f_{N}(mn)=\frac{1}{4\pi i}\cdot \int\limits_{c-i\infty}^{c+i\infty} \parentheses*{\frac{\sqrt{\pi}\abs{mn}}{\sqrt{N}}}^{-s} 
		\Gamma\parentheses*{\frac{s}{2}} \dd{s}
	\]
	and hence
	\begin{equation} \label{sigmaLs}
	\Sigma_{L}
		=  \frac{1}{4\pi i}\cdot {\int\limits_{c-i\infty}^{c+i\infty} \pi^{-s/2}\Gamma\parentheses*{\frac{s}{2}}N^{s/2} \Sigma_L(s) \dd{s}},
	\end{equation}
	where 
	\[
		\Sigma_L(s)\coloneqq \sum_{\substack{j\in\integers\setminus\braces{0} \\
			\Norm(j) \leq \delta^{-2} x^{\epsilon}
		}} 
		c_j \mathop{\sum\sum}_{\substack{ L<\Norm(m)\leq 2L \\ 0<\Norm(n)\leq x/L }} a_m(s)b_{n}(s) \e{\ImO{\omega}\parentheses{jmn\alpha}}
	\]
	with 
	$a_m(s)\coloneqq a_m\abs{m}^{-s}$ and $b_n(s)\coloneqq b_n\abs{n}^{-s}$.
	We set $c \coloneqq 1/\log x$.
	Then $N^{s/2} = O(1)$ and 
	\[
		a_m(s)\ll \abs{a_m} \quad \mbox{and} \quad b_n(s)\ll \abs{b_n}
	\]
	for all $s$ with $\Re s=c$ and $m$ and $n$ in the relevant summation ranges. 
	
	Now we estimate $\Sigma_L(s)$.
	In the following we tacitly assume that $x^{\mu} \leq L \leq x^{\mu+\kappa}$.
	By Cauchy's inequality,
	\begin{equation*} 
		\abs{\Sigma_{L}(s)}^{2}
		\ll \delta^{-2} x^{\epsilon} L \sum_{\substack{j\in\integers\setminus\braces{0} \\
			\Norm(j) \leq \delta^{-2} x^{\epsilon}
		}} 
		 \sum_{L<\Norm(m)\leq 2L}\ \abs[\Big]{
			\sum_{0<\Norm(n)\leq x/L} b_{n}(s) \e{\ImO{\omega}\parentheses{jmn\alpha}}
		}^{2},
	\end{equation*}
	which implies 
	\begin{equation} \label{afterCauchy}
		\abs{\Sigma_{L}(s)}^{2}
        	\ll \delta^{-2} x^{2\epsilon} L \sum_{\substack{k\in\integers\setminus\braces{0}\\ \Norm(k)\leq 2x^{\epsilon}L\delta^{-2}}}
		\abs[\Big]{
		\sum_{0<\Norm(n)\leq x/L} b_{n}(s)\e{\ImO{\omega}\parentheses{kn\alpha}}
		}^{2}.
	\end{equation}
	Next, we use the uniform bound
	$\boldsymbol{1}_{\braces{\Norm(k)\leq 2L'}}$ ($k\in\integers$)
	for $L'\coloneqq x^{\epsilon}L\delta^{-2}$ to extend the summation over $k$, getting
	\[
		\abs{\Sigma_{L}(s)}^{2}
		\ll \delta^{-2} x^{2\epsilon} L \sum_{k\in\integers} f_{L'}(k) \abs[\Big]{
			\sum_{0<\Norm(n)\leq x/L} b_{n}(s) \e{\ImO{\omega}\parentheses{kn\alpha}}
		}^{2}.
	\]
	Upon expanding the square in $\abs{\sum_{n}\ldots}^{2}$,
	\[
		\abs{\Sigma_{L}(s)}^{2}
		\ll \delta^{-2} x^{2\epsilon} L \!\!\mathop{\sum\sum}_{\substack{
			0 < \Norm(n_{1}) \leq x/L \\
			0 < \Norm(n_{2}) \leq x/L
		}}
		b_{n_{1}}(s) \overline{b_{n_{2}}(s)}  \sum_{k\in\integers}f_{L'}(k)
			\e{\ImO{\omega}\parentheses{k(n_{1}-n_{2})\alpha}}.
	\]
	The subsum with $n_{1} = n_{2}$ is $\ll \delta^{-4} x^{1+4\epsilon} L$. Thus, on writing $\ell = j (n_{1}-n_{2})$,
	\begin{equation}\label{eq:SigmaL2}
		\abs{\Sigma_{L}(s)}^{2}
		\ll \delta^{-4} x^{1+4\epsilon} L + \delta^{-2} x^{3\epsilon} L \Sigma_{L}',
	\end{equation}
	where
	\[
		\Sigma_{L}'
		= \sum_{\substack{
			\ell\in\integers\setminus\braces{0} \\
			\Norm(\ell) \leq 4 x / L
		}} \mathop{\sum\sum}_{\substack{
			0 < \Norm(n_{1}) \leq x/L \\
			0 < \Norm(n_{2}) \leq x/L \\
			\ell = n_{1}-n_{2}
		}} \abs[\Big]{
			\sum_{k\in\integers} f_{L'}(k) \e{\ImO{\omega}\parentheses{m\ell\alpha}}
		}.
	\]	
	By \cref{lem:PoissonResult-1}, it follows that
	\begin{align*}
		\Sigma_{L}' &
		\ll_{A,\epsilon,\omega} L'\sum_{\substack{
			\ell\in\integers\setminus\braces{0} \\
			\Norm(\ell)\leq 4x/L \\
			\norm{\ell\alpha}_{\omega}<x^{\epsilon}/\sqrt{L'}
		}} \mathop{\sum\sum}_{\substack{
			0 < \Norm(n_{1}) \leq x/L \\
			0 < \Norm(n_{2}) \leq x/L \\
			\ell = n_{1}-n_{2}
		}} 1 + x^{-A} \\ &
		\ll_{A,\epsilon,\omega} x^{1+\epsilon} \delta^{-2} \sum_{\substack{
			\Norm(\ell) \leq 4 x/ L \\
			\norm{\ell\alpha}_{\omega} < x^{\epsilon}/ \sqrt{L'}
		}} 1 + x^{-A}.
	\end{align*}
	By \cref{thm:Diophantine},
	\begin{align*} 
		\Sigma_{L}' &
		\ll_{C,A,\epsilon,\omega} x^{1+\epsilon} \delta^{-2} \parentheses{
			1 + xL^{-1}\abs{q}^{-2}
		} \parentheses{
			1 + x^{2\epsilon} L'^{-1} \abs{q}^{2}
		} + x^{-A}\\ &
		\ll x^{1+2\epsilon}\delta^{-2}\parentheses[\big]{ 1+xL^{-1}\abs{q}^{-2}+\delta^{2}L^{-1} \abs{q}^{2}+\delta^{2}xL^{-2} }.
	\end{align*}
	Combining this with~\cref{eq:SigmaL2} and taking square root, we obtain the estimate
	\begin{equation} \label{firstsigmaL}
	 \abs{\Sigma_L(s)} \ll x^{3\varepsilon}\parentheses[\big]{ \delta^{-2}x^{1/2}L^{1/2}+\delta^{-2}x\abs{q}^{-1}+\delta^{-1}x^{1/2}\abs{q}+
	 \delta^{-1}xL^{-1/2} }. 
	\end{equation}
	
	We may reverse the roles of $m$ and $n$ and estimate $\abs{\Sigma_L(s)}^2$ by 
	\[
		\abs{\Sigma_L(s)}^2\ll \delta^{-2} x^{1+2\epsilon} L^{-1} \sum_{\substack{k\in\integers\setminus\braces{0}\\ 
		\Norm(k)\leq 2x^{1+\epsilon}L^{-1}\delta^{-2}}}
		\abs[\Big]{
			\sum_{0<\Norm(m)\leq L} a_{m}(s)\e{\ImO{\omega}\parentheses{km\alpha}}
		}^{2}
	\]
	in place of \eqref{afterCauchy}.
	Then we can continue in a similar way as above. We arrive at the same estimate as in \eqref{firstsigmaL} but with $L$ replaced by $x/L$, i.e.
	\begin{equation} \label{secondsigmaL}
		\abs{\Sigma_L(s)} \ll x^{3\varepsilon} \parentheses[\big]{ \delta^{-2}xL^{-1/2}+\delta^{-2}x\abs{q}^{-1}+\delta^{-1}x^{1/2}\abs{q}+
		\delta^{-1}x^{1/2}L^{1/2} }.
	\end{equation}
	Using \eqref{firstsigmaL} if $L\le x^{1/2}$ and \eqref{secondsigmaL} if $L>x^{1/2}$, and recalling 
	that $x^{\mu}\le L\le x^{\mu+\kappa}$, we deduce that
	\[
		\abs{\Sigma_L(s)} \ll \begin{multlined}[t]
			x^{3\varepsilon}(\delta^{-2}\min\{x^{3/4},x^{(\mu+\kappa+1)/2}\}
			+ \delta^{-2}x\abs{q}^{-1} +{} \\ 
			\shoveleft[1cm]{
				+ \delta^{-1}x^{1/2}\abs{q}
				+ \delta^{-1}x^{1-\mu/2}
				+ \delta^{-1}x^{(\mu+\kappa+1)/2}).
			}
        \end{multlined}
    \]
	Using \eqref{sigmaLs} together with Stirling's approximation for the Gamma function, the same bound, up to a factor of $\log x$, 
	holds for $\Sigma_L$. 
	Plugging this into~\cref{eq:SigmaDyadic}, we obtain the assertion of the proposition.
\end{proof}

\subsection{Conclusion}
\begin{proof}[Proof of \cref{thm:ApproxWithPrimes}]
	We note that
	\[
		\lim_{R\to\infty} \sum_{\substack{ r\in\integers\\ \Norm(r)<R }} d_{4}(r\integers) \tilde{w}(r)
		\leq \lim_{R\to\infty} \sum_{\substack{ r\in\integers\\ \Norm(r)<R }} d_{4}(r\integers) w(r)
		\leq x^3
	\]
	for all sufficiently large $x$ depending on $\epsilon$. 
	(Mind though that the first two quantities depend implicitly on $x$ by means of the definitions of $\tilde{w}$, $w$, and $N$.) 
	Then, for $\mu\in(0,\frac{1}{2})$, choosing $\kappa=\frac{1}{2}$ and $M=2x^{\mu}$ in accordance with \cref{thm:HarmansSieve}, Propositions \ref{prop:Type-I-1} and
	\ref{prop:Type-II-1} give
	\begin{align*}
		\MoveEqLeft
		\abs{S(w,\sqrt{x}) - S(\tilde{w},\sqrt{x})} \\ &
		\ll_{C,\epsilon,\omega} \begin{multlined}[t]
			\delta^{2} N
			\cdot x^{7\epsilon} \bigl(
			\abs{q}^{2}x^{-1}
			+ \delta^{-2} \abs{q}^{-1} + \delta^{-1}\abs{q}x^{-1/2} +{}
			\\
			\delta^{-2} x^{-1/4} 
			+ \delta^{-1} x^{-\mu/2} + \delta^{-2} x^{-1+\mu} \bigr).
		\end{multlined}
	\end{align*}
	Upon taking $x = \abs{q}^{3}$ and $\mu=\frac{1}{4}$, we find
	\[
		\abs{S(w,\sqrt{x}) - S(\tilde{w},\sqrt{x})}
		\ll_{C,\epsilon,\omega} \delta^{2} N \cdot x^{-\epsilon}
	\]
	provided that $\tfrac{1}{2}\geq\delta \geq x^{-1/8+10\epsilon}$.
	By using \cref{cor:Detector}, the theorem follows.
\end{proof}

\section{Proof of the weighted version of Harman's sieve for \texorpdfstring{$\integers$}{O}}
\label{sec:HarmansSieveProof}

Before embarking on the proof of \cref{thm:HarmansSieve}, we record the following useful lemma:
\begin{lem}
	\label{lem:PerronVariant}
	For any two distinct real numbers $\rho,\gamma>0$ and $T\geq1$ one has
	\[
		\abs*{
			\ConditionalOne{\gamma<\rho}
			- \frac{1}{\pi} \int_{-T}^{T} e^{i\gamma t} \frac{\sin\parentheses{\rho t}}{t} \dd{t}
		} \ll \frac{1}{T \abs{\gamma-\rho}},
	\]
	where the implied constant is absolute.
\end{lem}
\begin{proof}
	See, for instance, \cite[Lemma~2.2]{harman2007primedetectingsieves}.
\end{proof}

\begin{proof}[Proof of \cref{thm:HarmansSieve}]
We follow~\cite{harman2007primedetectingsieves} quite closely. We assume that $\omega$ takes the values $w$ and $\tilde{w}$ and put $z = x^{\kappa}$.

We start by introducing some notation.
For each prime ideal of $\integers$ choose a generator $p\in\integers$ and let $\Primes_{\integers}$ be the set of all those $p$. For $Z\geq 0$ write
\[
	\Primes_{\integers}(Z) = \braces{p\in\Primes_{\integers} : \Norm(p) < Z}
\]
and
\[
	P_{\integers}(Z) = \prod_{p\in\Primes_{\integers}(Z)} p.
\]
We also need to introduce some $\integers$-version of the Möbius $\mu$ function: for a non-unit $d$, let $\mu(d)$ be defined as $(-1)^{r}$ if $d$ is the product of precisely $r$ non-associate prime elements and $\mu(d)=0$ otherwise. If $d$ is a unit, then put $\mu(d)=1$. Then, by inclusion--exclusion, we have
\begin{equation}\label{eq:InclusionExclusion}
	S(\omega,z)
	= \sum_{r\in\integers\setminus\braces{0}} \omega(r) \sum_{\substack{ m\mid P_{\integers}(z) \\ m\mid r }} \mu(m)
	= \sum_{m\mid P_{\integers}(z)} \mu(m) \sum_{n\in\integers\setminus\braces{0}} \omega(mn).
\end{equation}
On writing
\begin{equation}\label{eq:DeltaDef}
	\Delta(m) = \sum_{n\in\integers\setminus\braces{0}} (w(mn) - \tilde{w}(mn)),
\end{equation}
applying \cref{eq:InclusionExclusion} for $\omega=w$ and $\omega=\tilde{w}$ yields
\begin{equation}\label{eq:DiffErrorTerm}
	\begin{aligned}
		S(w,z) - S(\tilde{w},z) &
		= \braces[\bigg]{ \sum_{\substack{ m\mid P_{\integers}(z) \\ \Norm(m)<M }} + \sum_{\substack{ m\mid P_{\integers}(z) \\ \Norm(m)\geq M }}} \mu(m) \Delta(m) \\ &
		= S_{\mathrm{I}} + S_{\mathrm{II}}, \quad \text{say}.
	\end{aligned}
\end{equation}
By~\cref{eq:Harman:Type:I} with $a_m = \mu(m) \ConditionalOne{ m \mid P_{\integers}(z) }$ we infer $\abs{S_{\mathrm{I}}} \leq Y$. Therefore, to prove the theorem, it remains to establish that
\begin{equation}\label{eq:S_II:Bound}
	\abs{S_{\mathrm{II}}} \ll Y (\log(xX))^3.
\end{equation}

The next step is to arrange $S_{\mathrm{II}}$ into subsums according to the \enquote{size} of the prime factors in $m$ (where $m$ is the summation variable from~\cref{eq:DiffErrorTerm}). To have some such notion of size, fix some total order $\prec$ on $\Primes_{\integers}(z)$ such that $\Norm(p_2) \leq \Norm(p_1)$ whenever $p_2 \prec p_1$. (Clearly many such orders exist, but the precise choice must not concern us.) Moreover, for $p\in\Primes_{\integers}(z)$, let
\[
	\Pi(p) = \prod_{q \prec p} q.
\]
Now take $g\colon\integers\to\CC$ to be any function with $g(m) = g(\tilde{m})$ whenever $m$ and $\tilde{m}$ are associates. Then, we may group the terms of the sum
\[
	S = \sum_{m\mid P_{\integers}(z)} \mu(m) g(m)
\]
according to the largest prime factor $p_1$ of $m$ (w.r.t.~$\prec$):
\begin{equation}\label{eq:BuchstabTypeIdentity:I}
	S = g(1) - \sum_{p_{1}\in\mathbb{P}_{\integers}(z)} \sum_{d\mid\Pi(p_{1})} \mu(d) g(p_{1}d).
\end{equation}
Evidently, the process giving~\cref{eq:BuchstabTypeIdentity:I} also works if $P_{\integers}(z)$ is replaced by $\Pi(p)$; for any $r\in\integers$ one has
\begin{equation}\label{eq:BuchstabTypeIdentity:II}
	\sum_{d\mid\Pi(p_{1})} \mu(d) g(rd)
	= g(r) - \sum_{p_{2}\prec p_{1}} \sum_{d\mid\Pi(p_{2})} \mu(d) g(rp_{2}d).
\end{equation}
Minding the inner most sum on the right hand side above, it is obvious that the above identity can be iterated if so desired. To describe for which sub-sums iteration is beneficial, we let
\begin{align*}
	\mathbb{P}_{\integers}(z) &
	= \braces{ p_{1} \in \mathbb{P}_{\integers}(z) : \Norm(p_{1}) > x^{\mu} } \cupdot \braces{ p_{1} \in \mathbb{P}_{\integers}(z) : \Norm(p_{1}) \leq x^{\mu} } \\ &
	= \mathscr{P}_{1} \cupdot \mathscr{Q}_{1},
	\quad \text{say},
\end{align*}
and, inductively for $s=2,3,\ldots$,
\begin{align*}
	\mathscr{Q}_{s}' &
	= \braces{
		\parentheses{p_{1},\ldots,p_{s-1},p_{s}} \in \parentheses{\mathbb{P}_{\integers}(z)}^{s} :
		p_{s} \prec p_{s-1},\, \parentheses{p_{1},\ldots,p_{s-1}} \in \mathscr{Q}_{s-1}
	} \\ &
	= \mathscr{P}_{s} \cupdot \mathscr{Q}_{s},
\end{align*}
where
\begin{align*}
	\mathscr{P}_{s} &
	= \braces{
		\parentheses{p_{1},\ldots,p_{s-1},p_{s}} \in \mathscr{Q}_{s}' : N\parentheses{p_{1}\cdots p_{s-1}p_{s}} > x^{\mu}
	},\\
	\mathscr{Q}_{s} &
	= \braces{
		\parentheses{p_{1},\ldots,p_{s-1},p_{s}} \in \mathscr{Q}_{s}' : N\parentheses{p_{1} \cdots p_{s-1}p_{s}} \leq x^{\mu}
	}.
\end{align*}
Assuming that $g$ vanishes on arguments $r$ with $\Norm(r)\leq x^{\mu}$, and on applying~\cref{eq:BuchstabTypeIdentity:I} and~\cref{eq:BuchstabTypeIdentity:II},
\begin{align*}
	S &
	= - \braces[\bigg]{ \sum_{p_{1}\in\mathscr{P}_{1}} + \sum_{p_{1}\in\mathscr{Q}_{1}} } \sum_{d\mid\Pi(p_{1})} \mu(d) g(p_{1}d) \\ &
	= -\sum_{p_{1}\in\mathscr{P}_{1}} \sum_{d\mid\Pi(p_{1})} \mu(d)g(p_{1}d)
	\begin{multlined}[t]
		+ \sum_{(p_{1},p_{2})\in\mathscr{P}_{2}} \sum_{d\mid\Pi(p_{2})} \mu(d) g(p_{1}p_{2}d) \\
		+ \sum_{(p_{1},p_{2})\in\mathscr{Q}_{2}} \sum_{d\mid\Pi(p_{2})} \mu(d) g(p_{1}p_{2}d). \hfill\phantom.
\end{multlined}
\end{align*}
On iterating this process---always applying~\cref{eq:BuchstabTypeIdentity:II} to the $\mathscr{Q}$-part---it transpires that
\[
	S = \!\begin{multlined}[t]
		\sum_{s\leq t} \parentheses{-1}^{s} \sum_{(p_{1},p_{2},\ldots,p_{s}) \in \mathscr{P}_{s}} \sum_{d\mid\Pi(p_{s})} \mu(d) g(p_{1}p_{2} \cdots p_{s}d) +{} \\
		+ \parentheses{-1}^{t} \sum_{(p_{1},p_{2},\ldots,p_{t}) \in \mathscr{Q}_{t}} \sum_{d\mid\Pi(p_{t})} \mu(d) g(p_{1}p_{2} \cdots p_{t}d)
	\end{multlined}
\]
for any $t\in\NN$. Since the product of $t$ prime elements has norm $\geq2^{t}$, we have
\[
	\mathscr{Q}_{t} = \emptyset
	\quad\text{for}\quad
	t > \frac{\mu}{\log2} \log x.
\]
Hence,
\[
	S = \sum_{s\leq t} \parentheses{-1}^{s} \sum_{(p_{1},p_{2},\ldots,p_{s}) \in \mathscr{P}_{s}} \sum_{d\mid\Pi(p_{s})} \mu(d) g(p_{1}p_{2}\cdots p_{s}d)
\]
for (say)
\begin{equation}\label{eq:ChoiceOf:t}
	t = \floor{ (\log x) / \log 2} + 1 \ll \log x.
\end{equation}

We apply this to $S_{\mathrm{II}}$ with $g(m) = \Delta(m) \ConditionalOne{\Norm(m)\geq M}$.
Note that, since $M>x^{\mu}$, we have $g(r)=0$ for all $r$ with $\Norm(r)\leq x^{\mu}$, as was assumed in the above arguments.
Thus,
\begin{equation}\label{eq:S_II:Decomposition}
	S_{\mathrm{II}} = \sum_{s\leq t} \parentheses{-1}^{s} S_{\mathrm{II}}(s),
\end{equation}
where
\[
	S_{\mathrm{II}}(s) = \sum_{\substack{(p_{1},\ldots,p_{s})\in\mathscr{P}_{s} \\ m\coloneqq p_{1}\cdots p_{s} }} \sum_{\substack{ d\mid\Pi(p_{s}) \\ \Norm(md)\geq M }} \mu(d) \Delta(md).
\]
Another application of~\cref{eq:BuchstabTypeIdentity:II} gives
\begin{align}
	\nonumber
	S_{\mathrm{II}}(s) &
	= \sum_{\substack{ (p_{1},\ldots,p_{s}) \in \mathscr{P}_{s} \\ m\coloneqq p_{1}\cdots p_{s} \\ \Norm(m)\geq M }} \Delta(m) - \sum_{\substack{(p_{1},\ldots,p_{s})\in\mathscr{P}_{s} \\ m\coloneqq p_{1}\cdots p_{s} }} \sum_{p\prec p_{s}} \sum_{\substack{d\mid\Pi(p) \\ \Norm(mpd)\geq M }} \mu(d) \Delta(mpd) \\ &
	\label{eq:S_II:s:Decomposition}
	= S_{\mathrm{II},1}(s) - S_{\mathrm{II},2}(s),
	\quad\text{say}.
\end{align}
Given $m=p_{1}\cdots p_{s-1}p_{s}$ with
\[
	\parentheses{p_{1},\ldots,p_{s-1},p_{s}} \in \mathscr{P}_{s}
	\quad\text{and}\quad
	\parentheses{p_{1},\ldots,p_{s-1}} \in \mathscr{Q}_{s-1},
\]
and noting that $\Norm(p_{s})\leq \Norm(p_{1})<z=x^{\kappa}$, we have
\[
	x^{\mu}
	< \Norm(m)
	= \Norm\parentheses{ p_{1}\cdots p_{s-1} } \Norm(p_{s})
	< x^{\mu}x^{\kappa}.
\]
Using this, we find that $S_{\mathrm{II},1}(s)$ can be expressed as
\[
	\mathop{\sum\sum}_{m,n\in\integers\setminus\braces{0}} a_m (w(mn) - \tilde{w}(mn)),
\]
where the coefficients
\[
	a_m = \ConditionalOne{\Norm(m)\geq M} \, \ConditionalOne{p_{1}\cdots p_{s} : \parentheses{p_{1},\ldots,p_{s}} \in \mathscr{P}_{s}}(m)
\]
are only supported on $m$ with $x^{\mu}<\Norm(m)<x^{\mu+\kappa}$. Hence, by~\cref{eq:Harman:Type:II},
\begin{equation}\label{eq:S_II_1:s:Bound}
	\abs{S_{\mathrm{II},1}(s)} \leq Y.
\end{equation}
Moving on to $S_{\mathrm{II},2}(s)$, we expand the definition~\cref{eq:DeltaDef} of $\Delta$, getting
\[
	S_{\mathrm{II},2}(s) = S_{\mathrm{II},2}(s,w) - S_{\mathrm{II},2}(s,\tilde{w}),
\]
where
\begin{align*}
	S_{\mathrm{II},2}(s,\omega) &
	= \sum_{\substack{
		(p_{1},\ldots,p_{s}) \in \mathscr{P}_{s} \\
		m \coloneqq p_{1}\cdots p_{s}
	}} \sum_{p\prec p_{s}} \sum_{\substack{
		d\mid\Pi(p) \\
		\Norm(mpd) \geq M
	}} \mu(d) \sum_{\ell\in\integers\setminus\braces{0}} \omega(m\ell pd) \\ &
	= \sum_{\substack{
		(p_{1},\ldots,p_{s}) \in \mathscr{P}_{s} \\
		m\coloneqq p_{1}\cdots p_{s}
	}} \adjustlimits \sum_{n\in\integers\setminus\braces{0}} \sum_{p\prec p_{s}} \mathop{\sum\sum}_{\substack{
		d\mid\Pi(p) \\
		\ell pd=n \\
		\mathclap{ \Norm(mpd) \geq M }
	}} \mu(d) \omega(mn).
\end{align*}
In order to apply~\cref{eq:Harman:Type:II}, we must disentangle the variables $m$ and $n$ in the above summation. To this end, split
\begin{equation}\label{eq:PrimeIneqSplit}
	\sum_{p\prec p_{s}}
	= \sum_{\substack{ p\prec p_{s} \\ \Norm(p)=\Norm(p_{s}) }} + \sum_{\substack{ p\prec p_{s} \\ \Norm(p)<\Norm(p_{s}) }}
\end{equation}
to obtain a decomposition
\begin{equation}\label{eq:S_II_2:Decomposition}
	S_{\mathrm{II},2}(s,\omega)
	= S_{\mathrm{II},2}^{=}(s,\omega) + S_{\mathrm{II},2}^{<}(s,\omega),
	\quad \text{say}.
\end{equation}
For $S_{\mathrm{II},2}^{<}(s,\omega)$ we have
\[
	S_{\mathrm{II},2}^{<}(s,\omega)
	= \sum_{\substack{
		(p_{1},\ldots,p_{s})\in\mathscr{P}_{s} \\
		m\coloneqq p_{1}\cdots p_{s}
	}} \adjustlimits \sum_{n\in\integers\setminus\braces{0}} \sum_{p\in\mathbb{P}_{\integers}(z)} \mathop{\sum\sum}_{\substack{d\mid\Pi(p) \\ \ell pd=n }} \mu(d) \chi(m,d,p,p_{s}) \omega(mn),
\]
where
\[
	\chi(m,d,p,p_{s}) = \ConditionalOne{\Norm(mpd)\geq M} \ConditionalOne{\Norm(p)<\Norm(p_{s})},
\]
and the sum $S_{\mathrm{II},2}^{=}(s,\omega)$ can be expressed similarly, but needs a little more care: by basic ramification theory, the first summation on the right hand side of~\cref{eq:PrimeIneqSplit} contains at most one term and we shall write $\mathscr{P}_{s}'$ for the set of $(p_{1},\ldots,p_{s})\in\mathscr{P}_{s}$ for which there is such a term, that is, some $p\prec p_{s}$ with $\Norm(p)=\Norm(p_{s})$.
Furthermore, let $\mathbb{P}_{\integers}(z)'$ denote the set of the $p$'s just mentioned, i.e.,
\[
	\mathbb{P}_{\integers}(z)'
	= \braces{
		p\in\mathbb{P}_{\integers}(z) :
		\exists p_{s} \text{ s.t. } p \prec p_{s},\,
		\Norm(p) = \Norm(p_{s})
	}.
\]
Thus,
\[
	S_{\mathrm{II},2}^{=}(s,\omega)
	= \sum_{\substack{(p_{1},\ldots,p_{s})\in\mathscr{P}_{s}' \\ m\coloneqq p_{1}\cdots p_{s} }} \adjustlimits\sum_{n\in\integers\setminus\braces{0}}\sum_{p\in\mathbb{P}_{\integers}(z)'} \mathop{\sum\sum}_{\substack{ d\mid\Pi(p) \\ \ell pd=n }} \mu(d) \tilde{\chi}(m,d,p,p_{s}) \omega(mn),
\]
where
\begin{align}
	\nonumber
	\tilde{\chi}(m,d,p,p_{s}) &
	= \ConditionalOne{\Norm(mpd)\geq M} \ConditionalOne{\Norm(p)=\Norm(p_{s})} \\ &
	\label{eq:ChiTildeDecomposition}
	= \ConditionalOne{\Norm(mpd)\geq M} \ConditionalOne{\Norm(p)\leq \Norm(p_{s})} - \chi(m,d,p,p_{s}).
\end{align}

To disentangle $m$-dependent quantities ($\Norm(m)$ and $\Norm(p_s)$) from $n$-\hspace{0pt}dependent quantities ($\Norm(pd)$ and $\Norm(p)$) in the above, we employ
\cref{lem:PerronVariant}.
We pick some real number $\varrho$ (depending only on $M$) with $\abs{\varrho}\leq\frac{1}{2}$ such that $\braces{M+\varrho}=\frac{1}{2}$ and for $m,p,d\in\integers$ the condition $\Norm(mpd)\geq M$ is equivalent to $\log N\parentheses{mpd}\geq\log\parentheses{M+\varrho}$. Then
\[
	\abs{ \log N\parentheses{mpd} - \log\parentheses{M+\varrho} }
	\geq \log\frac{x+1}{x+\frac{1}{2}}
	\geq \frac{1}{3x}.
\]
Therefore, \cref{lem:PerronVariant} shows that
\[
	\ConditionalOne{\Norm(mpd)\geq M}
	= 1 - \frac{1}{\pi} \int_{-T}^{T} \parentheses{N\parentheses{mpd}}^{it} \sin\parentheses{t\log\parentheses{M+\varrho}} \frac{\dd{t}}{t} + O(x/T)
\]
for every $T\geq 1$.
Similarly,
\begin{gather*}
	\ConditionalOne{\Norm(p)<\Norm(p_{s})}
	= \frac{1}{\pi} \int_{-T}^{T} e^{\frac{it}{2}} e^{it\Norm(p)} \sin\parentheses{t\Norm(p_{s})} \frac{\dd{t}}{t} + O(1/T),\\
	\ConditionalOne{\Norm(p)\leq \Norm(p_{s})}
	= \frac{1}{\pi} \int_{-T}^{T} e^{-\frac{it}{2}} e^{it\Norm(p)} \sin\parentheses{t\Norm(p_{s})} \frac{\dd{t}}{t} + O(1/T).
\end{gather*}
Thus,
\begin{equation}\label{eq:S_II_2<:Integrals}
	S_{\mathrm{II},2}^{<}(s,\omega)
	= \!\begin{multlined}[t]
		\frac{1}{\pi} \int_{-T}^{T}
			\mathop{\sum\sum}_{m,n\in\integers\setminus\braces{0}} a_m(t)b_n(t) \omega(mn) \frac{\dd{t}}{t} -{} \\
		\shoveleft[.5cm]{ - \frac{1}{\pi^{2}} \int_{-T}^{T} \int_{-T}^{T}
			\mathop{\sum\sum}_{m,n\in\integers\setminus\braces{0}} a_m(t,\tau) b_n(t,\tau) \omega(mn)
		\frac{\dd{\tau}}{\tau} \frac{\dd{t}}{t} +{} } \\
		\shoveleft[.5cm]{ + O\parentheses*{ \frac{x}{T} + \frac{1}{T} \int_{-T}^{T} \abs{\sin\parentheses{ \tau \log\parentheses{M+\varrho} }} \frac{\dd{\tau}}{\abs{\tau}} } \times{} } \\
		\shoveleft[1.5cm]{ \times O\parentheses[\bigg]{
			\sum_{\substack{(p_{1},\ldots,p_{s})\in\mathscr{P}_{s} \\ m\coloneqq p_{1}\cdots p_{s}}} \sum_{n\in\integers\setminus\braces{0}} \sum_{p\in\mathbb{P}_{\integers}(z)} \mathop{\sum\sum}_{\substack{d\mid\Pi(p) \\ \ell pd=n}} \omega(mn)
		}, \hfill }
	\end{multlined}
\end{equation}
with coefficients
\begin{equation}\label{eq:TypeII:Coefficients:A}
	\begin{gathered}
		a_{m}(t) = \begin{cases}
			\sin\parentheses{t\Norm(p_{s})} & \text{if } \exists (p_{1},\ldots,p_{s}) \in \mathscr{P}_{s} \text{ s.t. } m = p_{1}\cdots p_{s}, \\
			0 & \text{otherwise},
		\end{cases} \\
		b_{n}(t) = \sum_{p\in\mathbb{P}_{\integers}(z)} \mathop{\sum\sum}_{\substack{d\mid\Pi(p) \\ \ell pd=n \neq 0 }} e^{\frac{it}{2}} e^{it\Norm(p)} \mu(d),
	\end{gathered}
\end{equation}
as well as
\begin{equation}\label{eq:TypeII:Coefficients:B}
	\begin{gathered}
		a_{m}(t,\tau) = a_{m}(t) \parentheses{N\parentheses m}^{i\tau} \sin\parentheses{\tau\log\parentheses{M+\varrho}}, \\
		b_{n}(t,\tau) = \sum_{p\in\mathbb{P}_{\integers}(z)} \mathop{\sum\sum}_{\substack{d\mid\Pi(p) \\ \ell pd=n \neq 0 }} e^{\frac{it}{2}} e^{it\Norm(p)} \mu(d) \parentheses{N\parentheses{pd}}^{i\tau}.
	\end{gathered}
\end{equation}

We proceed by gathering some intermediate information before applying~\cref{eq:Harman:Type:II}: in the definition of the coefficients $b_{n}$, neither of the summations over $p$ and $d$ includes associates. Thus,
\[
	\abs{b_{n}(t)},
	\abs{b_{n}(t,\tau)}
	\leq d(n\integers).
\]
For the other coefficients we always have
\[
	\abs{a_{m}(t)},
	\abs{a_{m}(t,\tau)}
	\leq 1,
\]
yet if $t$ and $\tau$ are small, one can (and must) do better: indeed, 
\begin{equation} \label{eq:Small:t:IntegrandBound}
	\abs{a_m(t)}\le \min\{1,\abs{t}\delta_1\}
	,\quad 
	\abs{a_m(t,\tau)}\le \min\{1,\abs{t}\delta_1,\abs{\tau}\delta_2,\abs{t\tau}\delta_1\delta_2\}, 
\end{equation}
where 
$$
\delta_1 \coloneqq x^{1/2} \quad \text{and} \quad \delta_2 \coloneqq \log\parentheses*{x+\frac{1}{2}}. 
$$
In view of this, we must deal with functions $f\colon\RR\times(0,1)\to\RR$ of the shape
\[
	f(t,\delta) = \begin{cases}
		\delta t & \text{if } \abs{t} \leq \delta^{-1}, \\
		1        & \text{otherwise}
	\end{cases}
\]
and their integrals
\begin{equation}\label{eq:PiecewiseIntegralBound}
	\int_{-T}^{T} f(t,\delta) \frac{\dd{t}}{\abs{t}}
	\ll \delta {\int_{0}^{\delta^{-1}} \dd{t}} + \abs*{
		\int_{\delta^{-1}}^{T} \frac{\dd{t}}{t}
	}
	\ll 1 + \abs{\log(T\delta)}.
\end{equation}
Lastly, we note that, by \cref{lem:AlgebraicFacts} and \cref{eq:Convergence},
\begin{equation}\label{eq:SieveTrivialBound}
	\begin{aligned}
		\MoveEqLeft[7]
		\sum_{\substack{(p_{1},\ldots,p_{s})\in\mathscr{P}_{s} \\ m\coloneqq p_{1}\cdots p_{s} }} \sum_{n\in\integers\setminus\braces{0}} \sum_{p\in\mathbb{P}_{\integers}(z)} \mathop{\sum\sum}_{\substack{d\mid\Pi(p) \\ \ell pd=n}} \omega(mn) \\ &
		\ll \sum_{r\in\integers\setminus\braces{0}} d_{4}(r\integers) \omega(r)
		\ll X.
	\end{aligned}
\end{equation}

Collecting what we have gathered so far, we may derive a bound for
\[
	\mathcal{E} = \abs{ S_{\mathrm{II},2}^{<}(s,w) - S_{\mathrm{II},2}^{<}(s,\tilde{w}) }
\]
as follows: after applying~\cref{eq:S_II_2<:Integrals} with $\omega=w$ and $\omega=\tilde{w}$, the $O(\ldots)$-terms are treated directly with~\cref{eq:SieveTrivialBound} and~\cref{eq:PiecewiseIntegralBound}, whereas for the rest one may apply~\cref{eq:Harman:Type:II}. Here it is important to use
\cref{eq:Small:t:IntegrandBound} for small $\abs{t}$ respectively $\abs{\tau}$ 
first---prior to applying~\cref{eq:Harman:Type:II}---and~\cref{eq:PiecewiseIntegralBound} then bounds the integrals.
Therefore, after some computations, we infer
\begin{equation}\label{eq:S_II_2<:Bound}
	\mathcal{E} \ll \!\begin{multlined}[t]
		Y \log(Tx) \parentheses{ 1 + \log\parentheses{T \log\parentheses{x+\tfrac{1}{2}}}} +{} \\
		\shoveleft[.5cm]{ + X T^{-1} \parentheses{x + \log\parentheses{T \log\parentheses{x+\tfrac{1}{2}}}}. \hfill }
	\end{multlined}
\end{equation}
Of course, the same arguments also apply to $S_{\mathrm{II},2}^{=}(s,\omega)$; in view of~\cref{eq:ChiTildeDecomposition} we have to apply them twice, but in both cases the coefficients corresponding to~\cref{eq:TypeII:Coefficients:A} and~\cref{eq:TypeII:Coefficients:B} obey the same bounds we used to derive~\cref{eq:S_II_2<:Bound}.
Consequently,~\cref{eq:S_II_2<:Bound} also holds with $S_{\mathrm{II},2}^{=}$ in place of $S_{\mathrm{II},2}^{<}$.
In total, recalling~\cref{eq:S_II:s:Decomposition},~\cref{eq:S_II_1:s:Bound} and~\cref{eq:S_II_2:Decomposition}, we have
\[
	\abs{S_{\mathrm{II}}(s)} \ll Y + \braces{\text{the bound from~\cref{eq:S_II_2<:Bound}}}
\]
and it transpires that choosing $T = xX$ suffices to yield a bound $\ll Y(\log(xX))^{2}$.
On plugging this into~\cref{eq:S_II:Decomposition} and recalling~\cref{eq:ChoiceOf:t}, we infer~\cref{eq:S_II:Bound}.
Hence, the theorem is proved.
\end{proof}


\end{document}